\documentclass[a4paper,leqno,10pt]{amsart}

\usepackage{epsf,graphicx,epsfig}
\usepackage{amscd}
\usepackage{amsmath,latexsym,amssymb,amsthm}
\usepackage[nospace,noadjust]{cite}
\usepackage{textcomp}
\usepackage{setspace,cite}
\usepackage{lscape,fancyhdr,fancybox}
\usepackage{stmaryrd}
\usepackage[all,cmtip]{xy}
\usepackage{tikz}
\usepackage{cancel}
\usetikzlibrary{shapes,arrows,decorations.markings}
\usepackage{graphicx}

\usepackage{color}
\usepackage{url}
\usepackage{enumerate}
\usepackage[mathscr]{euscript}
  \theoremstyle{plain}

\swapnumbers
    \newtheorem{thm}{Theorem}[section]
    \newtheorem{prop}[thm]{Proposition}
     
   \newtheorem{lemma}[thm]{Lemma}
    \newtheorem{corollary}[thm]{Corollary}
    
    \newtheorem{subsec}[thm]{}
\theoremstyle{definition}
    \newtheorem{defn}[thm]{Definition}
        \newtheorem{remark}[thm]{Remark}
    \newtheorem{exam}[thm]{Example}

\theoremstyle{remark}

\title{}
\author{}
\date{}
\usepackage{amssymb}

\usepackage{hyperref}
\hypersetup{
	colorlinks,
	citecolor=blue,
	filecolor=black,
	linkcolor=blue,
	urlcolor=black
}

\begin{document}

\title{Extension, deformation and categorification of $\text{AssDer}$ pairs}
%\author{Apurba Das \\ Indian Institute of Technology, Kanpur 208016, Uttar Pradesh, India \\ (e-mail: apurbadas348@gmail.com)}
%\author{Apurba Das and Ashis Mandal\\Department of Mathematics and Statistics,
%Indian Institute of Technology,\\ Kanpur 208016, Uttar Pradesh, India.\\apurbadas348@gmail.com~~~
%amandal@iitk.ac.in}
%\email{apurbadas348@gmail.com}

\author{Apurba Das}
\address{Department of Mathematics, Indian Institute of Technology, Kharagpur 721302, West Bengal, India.}
\email{apurbadas348@gmail.com}

\author{Ashis Mandal}
\address{Department of Mathematics and Statistics,
Indian Institute of Technology, Kanpur 208016, Uttar Pradesh, India.}
\email{amandal@iitk.ac.in}


\maketitle
\begin{abstract}
In this paper, we consider associative algebras equipped with derivations. A pair consisting of an associative algebra and a distinguished derivation is called an AssDer pair. We study central extensions and formal one-parameter deformations of AssDer pairs in terms of cohomology. Finally, we define $2$-derivations on associative $2$-algebras and show that the category of associative $2$-algebras with $2$-derivations is equivalent to the category of $2$-term $A_\infty$-algebras with homotopy derivations.
\end{abstract}

%\noindent {\bf Keywords:} AssDer pair, Hochschild cohomology, Crossed extensions, Abelian extensions, Deformations, Homotopy derivations, Categorifications.

%\noindent {\bf 2010 MSC classifications:} 16E40, 16S80, 16W25, 18N25, 18N40.
%$17$A$30,$ $17$B$55$.

%\tableofcontents

\noindent
\thispagestyle{empty}

\maketitle

\section{Introduction}
Algebraic structures, such as Lie algebras and associative algebras are important in various areas of mathematics and physics. 
%To understand an algebraic structure, sometimes it is useful to know the group of algebra isomorphisms. 
Algebras are also useful via their derivations. If the algebra $A$ is a polynomial algebra in $n$ variables, some special types of derivations (locally nilpotent, locally finite etc.) are studied extensively in the literature.  
%For instance, to consider polynomial flows and also in connection to fourteenth problem of Hilbert as well (see in  \cite{Nowicki2002} and references).  
One can construct a homotopy Lie algebra out of a graded Lie algebra with a special derivation \cite{voro}. In \cite{coll} the authors use noncommuting derivations in an associative algebra to construct deformation formulas. Derivations are useful in control theory and gauge theories in QFT \cite{ref-1, ref-2}.
Algebras with derivations are also studied from an operadic point of view \cite{loday, doubek-lada}. 
 Recently, the authors in \cite{tang-f-sheng} considered Lie algebras equipped with derivations (also called LieDer pairs). More precisely, they study central extensions and deformations of LieDer pairs from a cohomological point of view.
 
 \medskip

 In this paper, we consider a pair of associative algebra $A$ with a distinguished derivation $\phi_A$. Such a pair $(A, \phi_A)$ of an associative algebra with a derivation is called an AssDer pair. 
 %This algebraic structure form a nice category and some invariants to study this new category are developed in this paper.
 It is also known as a differential algebra in the literature and is itself a significant object of study in areas such as differential Galois theory \cite{magid}. Derivations on unital, commutative associative algebras are also related to Lie-Rinehart algebras (cf. Proposition \ref{derivation-lie-rinehart}).
  %Recently, the authors in \cite{ban} defined Hochschild homology for AssDer pairs motivated by noncommutative geometry. 
  
\medskip  
  
  Here we construct a cohomology for an AssDer pair, study their central extensions, and formal one-parameter deformations.  We relate the cohomology of an AssDer pair with the cohomology of the corresponding commutator LieDer pair. This cohomology might be a starting point to study cyclic theory for AssDer pairs. Additionally, we define and study $2$-derivations on associative $2$-algebras and relate them with homotopy derivations of $2$-term $A_\infty$-algebras.
  
  \medskip
  
  In section \ref{sec2}, we study representations and cohomology of AssDer pairs. Let $(A, \phi_A)$ be an AssDer pair. A representation of it consists of an $A$-bimodule $M$ together with a linear map $\phi_M$ which is compatible with the left and right actions of $A$ on $M$. It turns out that any AssDer pair is a representation of itself. Given a representation $(M, \phi_M)$, the pair $(M^*, - \phi_M^*)$ is also a representation, where $M^*$ is equipped with the $A$-bimodule structure dual to $M$ (cf. Proposition \ref{dual-birep}). Given a representation of an AssDer pair, one can construct a semidirect product AssDer pair (cf. Proposition \ref{semi-d}). Next, we study the cohomology of an AssDer pair with coefficients in a representation. This cohomology is a follow-up to the Hochschild cohomology of the associative structure and a factor modified by the fixed derivation. Like Hochschild cohomology, we show that the cohomology of an AssDer pair with coefficients in itself carries a degree $-1$ graded Lie bracket (cf. Proposition \ref{deg-min}). 
Next, we construct a functor $U: \mathrm{\bf LieDer} \rightarrow \mathrm{\bf AssDer}$ from the category of LieDer pairs to the category of AssDer pairs.
  
  \medskip

In Sections \ref{sec-ext}, we study extensions of an AssDer pair by a trivial AssDer pair, called central extensions. We show that isomorphism classes of central extensions are classified by the second cohomology of the AssDer pair with coefficients in the trivial representation (cf. Theorem \ref{central-fin-thm}). Next, we study extensions of a pair of derivations in a central extension of associative algebras. Given a central extension of associative algebras $0 \rightarrow M \xrightarrow{i} \hat{A} \xrightarrow{p} A \rightarrow 0$ and a pair of derivations $(\phi_A, \phi_M) \in \mathrm{Der}(A) \times \mathrm{Der}(M)$, we associate a second cohomology class in the Hochschild cohomology of $A$ with trivial representation $M$ (cf. Proposition \ref{obs-class-cen}), called the obstruction class. When this cohomology class is null, the pair of derivations $(\phi_A, \phi_M)$ is extensible to a derivation $\phi_{\hat{A}} \in \mathrm{Der}(\hat{A})$ which makes the above sequence into an exact sequence of AssDer pairs (cf. Theorem \ref{thm-obs-tri}). 
%Finally, we consider abelian extensions of an AssDer pair by an arbitrary representation and show that isomorphism classes of such extensions are classified by the second cohomology group of the AssDer pair (cf. Theorem \ref{abl-fin-thm}).

\medskip

In Section \ref{sec-def}, we study formal one-parameter deformations of AssDer pairs following the classical approach of Gerstenhaber for associative algebras \cite{gers} and Nijenhuis-Richardson for Lie algebras \cite{nij-ric}. For this, we deform both the associative product as well as the given derivation. 
%Our study of deformations of AssDer pairs leads to the question of deformation quantization of Poisson algebras with Poisson derivations which we will trace in a future project (cf. Remark \ref{quan-poisson-der}).  
Here we will show that the vanishing of the second cohomology of an AssDer pair with coefficients in itself implies that the AssDer pair is rigid (cf. Theorem \ref{h-two}, Remark \ref{h-two-rem}). Given a finite order deformation of an AssDer pair, we associate a third cohomology class, called the obstruction class of the deformation (cf. Proposition \ref{obs-class}). If the class is trivial, then the deformation extends to a deformation of the next order (cf. Theorem \ref{obs-class-zero}). We also consider automorphisms of the deformed AssDer pair and their extensions (cf. Theorem \ref{auto-theorem}).

\medskip

Strongly homotopy associative algebras ($A_\infty$-algebras) were introduced by Stasheff to recognize loop spaces \cite{stasheff}.
In Section \ref{sec-hd}, we consider homotopy derivations on $A_\infty$-algebras whose underlying graded vector space is concentrated in degrees $0$ and $1$ \cite{loday, doubek-lada}. We denote the category of $2$-term $A_\infty$-algebras with homotopy derivations by $2\mathrm{\bf AssDer}_\infty$. Homotopy derivations on skeletal $A_\infty$-algebras are characterized by third cocycles of AssDer pairs (cf. Proposition \ref{skeletal-chr}) and `strict' homotopy derivations on strict $A_\infty$-algebras are characterized by crossed modules of AssDer pairs (cf. Proposition \ref{strict-chr}).
%We also find a generalization of this result which classify the $n$-th cohomology of AssDer pair by suitable $A_\infty$-algebras with homotopy derivations.

\medskip

In \cite{baez}, Baez and Crans introduced Lie $2$-algebras as the categorification of Lie algebras. They also showed that the category of $2$-term $L_\infty$-algebras and the category of Lie $2$-algebras are equivalent. This result has been extended to various other algebraic structures, including groups, Leibniz algebras and twisted associative algebras \cite{baez-lauda, baez-j, sheng-liu, das}. In section \ref{sec-cat}, we introduce the categorification of AssDer pairs. More precisely, we study AssDer pair structures on a $2$-vector space. We call such an object an AssDer $2$-pair. The category of AssDer $2$-pairs and morphisms between them is denoted by $\mathrm{\bf AssDer2}$. Finally, we show that the categories $2\mathrm{\bf AssDer}_\infty$ and $\mathrm{\bf AssDer2}$ are equivalent (cf. Theorem \ref{final-thrm}).

\medskip

In the whole paper, we assume that $\mathbb{K}$ is a fixed field of characteristic zero, all the vector spaces are over the field $\mathbb{K}$ and maps are $\mathbb{K}$-linear maps unless otherwise stated. Let $\mathsf{XAlg}$ be any given type of algebraic structure. Then by ${\bf 2XAlg_\infty}$, we denote the category of $2$-term $\mathsf{XAlg}_\infty$-algebras and by ${\bf XAlg2}$, we denote the category of $\mathsf{XAlg} ~2$-algebras.
  
\section{AssDer pairs}\label{sec2}
%In this section, we introduce representations and cohomology of AssDer pairs. We also relate this cohomology with the cohomology of LieDer pairs as introduced in \cite{tang-f-sheng}.

Let $A$ be an associative algebra. An  $A$-bimodule (also called a representation of $A$) is a vector space $M$ together with two linear maps $l : A \otimes M \rightarrow M, (a,m) \mapsto am$ and $r : M \otimes A \rightarrow M, (m,a) \mapsto ma$ satisfying
\begin{align*}
(ab) m = a(bm), \quad (am)b  = a(mb) ~~~ \text{ and } ~~~ (ma) b = m (ab),
\end{align*}
for all $a, b \in A$ and $m \in M$. It follows that $A$ is a bimodule over the associative algebra $A$ itself with the left and right actions given by the algebra multiplication. We call this the adjoint bimodule (representation).

A derivation on $A$ with values in the $A$-bimodule $M$ is given by a linear map $\phi : A \rightarrow M$ that satisfies
%\begin{align*}
$$\phi (a b) = \phi(a) b  +  a \phi (b), ~~~ \text{ for } a, b \in A.$$
%\end{align*}
%Derivations are $1$-cocycles in the Hochschild complex of the associative algebra $A$ with coefficients in $M$ \cite{hoch}. See also \cite{loday-book} for more details. 

Our main object in this paper is a pair $(A, \phi_A)$ in which $A$ is an associative algebra and $\phi_A: A \rightarrow A$ is a derivation on $A$ with values in the adjoint representation. Thus, $\phi_A$ satisfies
\begin{align*}
\phi_A ( a b ) = \phi_A (a ) b + a \phi_A (b),  ~~~ \text{ for } a, b \in A.
\end{align*}
Such a pair $(A, \phi_A )$ is called an AssDer pair. Here we give a few examples of AssDer pairs. 
\begin{exam}\label{ass-der-exam}
\begin{itemize}
\item[(i)] The notion of derivation is a generalization of the usual derivative of functions. For instance, if $A = \mathbb{K} [x_1, \ldots, x_n]$ is the polynomial algebra in $n$ variables, then for $1 \leq i \leq n$, the partial derivatives $\phi_i = \frac{\partial}{\partial x_i}$ are derivations on $A$. In fact, the space of derivations on $A$ is linearly spanned by $\{ \phi_i \}_{ 1 \leq i \leq n}.$
\item[(ii)] Any derivation on the space $C^\infty (M)$ of smooth functions on a manifold $M$ is given by a vector field. Therefore $(C^\infty(M), X)$ is an AssDer pair, for any vector field $X$ on $M$.
\item[(iii)] A  (non-commutative) Poisson algebra is an associative algebra $P$ together with a Lie bracket $\{ ~, ~\}$ on it which is a derivation on each entry for the associative product. It follows that if $(P, \{~, ~\})$ is a (noncommutative) Poisson algebra, then for any $a \in P$, the linear map $\phi_a = \{ a, ~ \}$ is a derivation for the associative product. Hence $(P, \phi_a)$ is an AssDer pair.
\item[(iv)] Let $V$ be a vector space and $d : V \rightarrow V$ be a linear map. Consider the reduced tensor algebra $\overline{T}(V) = \oplus_{n \geq 1} V^{\otimes n}$ with the concatenation product. The linear map $d$ induces a linear map 
$\overline{d} : \overline{T}(V) \rightarrow \overline{T}(V)$ by
\begin{align*}
\overline{d} ( v_1 \otimes \cdots \otimes v_n ) = \sum_{i=1}^n v_1 \otimes \cdots \otimes dv_i \otimes \cdots \otimes v_n.
\end{align*}
It is easy to verify the $\overline{d}$ is a derivation on $\overline{T}(V)$. Hence  $(\overline{T}(V), \overline{d})$ is an AssDer pair. Any derivation on $\overline{T}(V)$ arises in this way.
\end{itemize}
\end{exam}

\medskip

%\begin{remark}\label{ass-der-lie-r}
A {\em Lie-Rinehart algebra} \cite{hueb} is a triple $(A, L, \rho)$ where $A$ is a commutative associative algebra, $(L, [~,~]_L)$ is a Lie algebra with $L$ a left $A$-module and an $A$-module map $\rho: L \rightarrow \mathrm{Der}(A)$ which is a morphism of Lie algebras satisfying the following Leibniz rule
\begin{align*}
[X, aY]_L = a [X, Y]_L + \rho (X) (a) Y, \text{ for } X, Y \in L, a \in A. 
\end{align*}

Let $A$ be a commutative associative algebra and $\phi_A$ be a derivation on $A$. Then $\phi_A$ induces a Lie bracket on $A$ given by
\begin{align}\label{lie-from-der}
[a, b ]_{\phi_A} := a \phi_A (b) - \phi_A (a) b, ~~~ \text{ for } a, b \in A.
\end{align}
We denote this Lie algebra by $A_{\phi_A}.$ This Lie bracket additionally satisfies the Leibniz rule
\begin{align}\label{leib-ru}
[a, fb]_{\phi_A} = f [ a, b ]_{\phi_A} + \rho (a)(f) b, ~~~ \text{ for } a, b, f \in A,
\end{align}
where $\rho : (A, [~,~]_{\phi_A}) \rightarrow (\mathrm{Der}(A), [~, ~])$ is the $A$-linear Lie algebra morphism given by $\rho (a) = a \phi_A$. In other words, the triple $(A, A_{\phi_A}, \rho)$ is a Lie-Rinehart algebra. The next proposition says that any Lie-Rinehart algebra structure on $(A, A)$ arises by a derivation in this way when $A$ is unital. Thus, to better understand Lie-Rinehart algebra structures on $(A, A)$ one needs to know derivations on $A$.
%In Remark \ref{def-lr-alg}, we observe that this association is compatible with deformation theories of AssDer pairs and Lie-Rinehart algebras.
%\end{remark}

\begin{prop}\label{derivation-lie-rinehart}
Let $A$ be an unital, commutative associative algebra. There is a one-to-one correspondence between derivations on $A$ and Lie-Rinehart algebra structures on $(A, A)$.
\end{prop}

\begin{proof}
Given a derivation $\phi_A$, we have constructed a Lie-Rinehart algebra $(A, A_{\phi_A}, \rho)$. Conversely, let $(A, A_{\mathrm{Lie}}, \rho)$ be a Lie-Rinehart algebra. Define $\phi_A : A \rightarrow A$ by $\phi_A = \rho (1)$, where $1$ is the unit of $A$. By definition, $\phi_A$ is a derivation on $A$. Moreover, we have $\rho (a) = a \rho (1) = a \phi_A$. Hence, for $a, b \in A$,
\begin{align*}
[a, b ]_{\mathrm{Lie}} = b [a,1]_{\mathrm{Lie}} + \rho (a) (b) 
= - \rho_A (a) + a \phi_A (b) = [a, b]_{\phi_A}.
\end{align*}
Therefore, the Lie-Rinehart algebra $(A, A_{\mathrm{Lie}}, \rho)$ is obtained from the derivation $\phi_A$. Finally, the above two correspondences are inverses to each other.
\end{proof}

\begin{remark}
In this remark, we will show that the Witt algebra is obtained from a derivation in Laurent polynomial algebra. Let $A = \mathbb{K}[x, x^{-1}]$ be the Laurent polynomial algebra. Consider the derivation $\phi_A : A \rightarrow A$ given by $\phi_A (x^n ) = - n x^{n-1}$, for $n \in \mathbb{Z}$. It follows from (\ref{lie-from-der}) that $A = \mathbb{K}[x, x^{-1}]$ carries a Lie algebra structure given by
\begin{align}\label{laurent-pol}
[x^m, x^n ] = - x^m ( n x^{n-1}) + (m x^{m-1}) x^n = (m-n) x^{m+n-1}, \text{ for } x^m, x^n \in A, m, n \in \mathbb{Z}. 
\end{align}
Consider the basis $\{ l_n \}_{n \in \mathbb{Z}}$ for $A$, where $l_n = x^{n+1}$, $n \in \mathbb{Z}$. Then the Lie bracket (\ref{laurent-pol}) reads as $[l_m, l_n] = (m-n)~ l_{m+n}$, for $m,n\in \mathbb{Z}$. This is precisely the Witt algebra structure on $A = \text{ span }\{l_n\}_{n \in \mathbb{Z}}$.
\end{remark}

\begin{defn}
Let $(A, \phi_A)$ and $(B, \phi_B)$ be two AssDer pairs. A morphism between them consists of an algebra map $f: A \rightarrow B$ that commutes with derivations, i.e.
%\begin{align*}
$f \circ \phi_A = \phi_B \circ f.$
%\end{align*}
\end{defn}
We denote the category of AssDer pairs together with morphisms between them by {\bf AssDer}.

Let $(V, d)$ be a vector space together with a linear map. The free AssDer pair over $(V, d)$ is an AssDer pair $(\mathcal{F}(V), \phi_{\mathcal{F}(V)} )$ equipped with a linear map $i : V \rightarrow \mathcal{F}(V)$ that satisfies $\phi_{\mathcal{F}(V)} \circ i = i \circ d$ and the following universal condition holds:
for any AssDer pair $(A, \phi_A)$ and a linear map $f : V \rightarrow A$ satisfying $\phi_A \circ f =  f \circ d$, there exists a unique AssDer pair morphism $\widetilde{f} : (\mathcal{F}(V), \phi_{\mathcal{F}(V)} ) \rightarrow (A, \phi_A)$ such that $\widetilde{f} \circ i = f.$

It follows that the free AssDer pair over $(V, d)$ is well-defined up to a unique isomorphism.

\begin{prop}
Let $(V, d)$ be a vector space together with a linear map. Then $(T(V), \overline{d})$ $(\mathrm{resp.} (\overline{T}(V), \overline{d}))$ equipped with the inclusion map $i$ is free unital (resp. nonunital) AssDer pair over $(V, d).$
\end{prop}

\subsection{Representations and cohomology of AssDer pairs}

\begin{defn}
Let $(A, \phi_A)$ be an AssDer pair. A left module over it consists of a pair $(M, \phi_M)$ in which $M$ is a left $A$-module and $\phi_M : M \rightarrow M$ is a linear map satisfying
\begin{align}
\phi_M ( a m ) = \phi_A (a) m + a \phi_M (m), ~ \text{ for } a \in A, m \in M. \label{rep-1-eqn}
\end{align}
\end{defn}
Similarly, a right module over  $(A, \phi_A)$ is a pair $(M, \phi_M)$ in which $M$ is a right $A$-module and $\phi_M : M \rightarrow M$ is a linear map satisfying
\begin{align}
\phi_M (m a ) = \phi_M (m) a + m \phi_A (a), ~ \text{ for } a \in A, m \in M. \label{rep-2-eqn}
\end{align}
A bimodule (representation) over $(A, \phi_A)$ is a pair $(M, \phi_M)$ which is both a left module and a right module over $(A, \phi_A)$ and $M$ is an $A$-bimodule, i.e. $(am)b = a (mb),$ for all $a, b \in A$ and $m \in M$. It follows that the AssDer pair $(A, \phi_A)$ is a representation of itself for the adjoint bimodule structure on $A$.

%A representation of it consists of a pair $(M, \phi_M)$ in which $M = (M, l, r)$ is an $A$-bimodule and $\phi_M : M \rightarrow M$ is a linear map satisfying
%\begin{align}
%\phi_M ( a m ) = \phi_A (a) m + a \phi_M (m), \label{rep-1-eqn}\\
%\phi_M (m a ) = \phi_M (m) a + m \phi_A (a). \label{rep-2-eqn}
%\end{align}

\begin{prop}\label{dual-birep}
Let $(M, \phi_M)$ be a representation of the AssDer pair $(A, \phi_A)$. Then $(M^*, - \phi_M^*)$ is also a representation of $(A, \phi_A)$ where the $A$-bimodule structure on $M^*$ is given by
\begin{align*}
~&A \otimes M^* \rightarrow M^* \qquad \qquad \qquad  M^* \otimes A \rightarrow M^* \\
~&(a f)(m) = f (ma ) ~~~ \qquad \qquad \quad ( fa) (m) = f ( am ),
\end{align*}
for $a \in A, f \in M^*$ and $m \in M.$
\end{prop}

\begin{proof}
The fact that $M^*$ is an $A$-bimodule is standard \cite{loday-book}. To verify that $(M^*, - \phi_M^*)$ is a representation of the AssDer pair, we observe that 
\begin{align*}
\langle - \phi_M^* ( af) , m \rangle = \langle af , - \phi_M (m) \rangle = f ( - \phi_M (m) a ) 
=~& f ( m \phi_A (a) ) - f ( \phi_M (ma))  \qquad (\text{by }~~ (\ref{rep-2-eqn})) \\
=~& \langle \phi_A (a) f , m \rangle - \langle a \phi_M^* (f), m \rangle.
\end{align*}
This shows that $- \phi_M^* ( af) = \phi_A (a) f + a ( - \phi_M^* (f)).$ Similarly, we have
\begin{align*}
\langle - \phi_M^* ( fa ) , m \rangle =  \langle fa, - \phi_M (m) \rangle = f ( - a \phi_M (m)) 
=~& f ( \phi_A (a) m) - f ( \phi_M (am))  \qquad (\text{by }~~ (\ref{rep-1-eqn})) \\
=~& \langle f \phi_A (a), m \rangle - \langle \phi_M^*(f) a, m \rangle.
\end{align*}
This shows that $- \phi_M^* ( fa ) = - \phi_M^*(f) a + f \phi_A (a)$.
\end{proof}

It follows from the above Proposition, $(A^*, - \phi_A^*)$ is a representation of the AssDer pair $(A, \phi_A)$. This is called the coadjoint representation of the AssDer pair $(A, \phi_A).$

Given an associative algebra and a bimodule over it, one can construct a semidirect product associative algebra \cite{loday-book}. The following result generalizes it to AssDer pairs. The proof is straightforward.

\begin{prop}\label{semi-d}
Let $(A, \phi_A)$ be an ${AssDer}$ pair and $(M, \phi_M)$ be a representation of it. Then $(A \oplus M, \phi_A \oplus \phi_M)$ is an ${AssDer}$ pair where the associative structure on $A \oplus M$ is given by the semi-direct product
\begin{align*}
(a \oplus m) \cdot (b \oplus n) = (ab \oplus an + mb).
\end{align*}
\end{prop}
%\begin{proof}
%It is known that $A \oplus M$ with the above product is an associative algebra \cite{loday-book}. Thus, it is enough to show that $\phi_A  \oplus \phi_M: A \oplus M \rightarrow A \oplus M$ is a derivation. Observe that
%\begin{align*}
%(\phi_A  \oplus \phi_M) ( (a \oplus m) \cdot (b \oplus n) ) 
%&= ( \phi_A (ab) \oplus \phi_M (an) + \phi_M (mb)  ) \\
%&= (\phi_A (a) b + a \phi_A (b) \oplus \phi_A (a)n + a \phi_M (n) + \phi_M (m) b + m \phi_A (b) ) \\
%&= ( \phi_A (a) b \oplus \phi_A (a) n + \phi_M (m) b) ~+~ ( a \phi_A(b) \oplus a \phi_M (n) + m \phi_A (b) ) \\
%&= (\phi_A (a) \oplus \phi_M (m)) \cdot (b\oplus n) ~+~ (a\oplus m) \cdot (\phi_A(b) \oplus \phi_M (n) ) \\
%&= (\phi_A \oplus \phi_M )(a \oplus m) \cdot (b \oplus n) ~+~ (a \oplus m) \cdot (\phi_A \oplus \phi_M) (b \oplus n). 
%\end{align*}
%Hence the proof.
%\end{proof}

Let $A$ be an associative algebra and $M$ be an $A$-bimodule. Then the Hochschild cohomology of $A$ with coefficients in $M$ is given as follows \cite{hoch}. The $n$-th cochain group $C^n (A, M)$ is given by $C^n(A, M) := \text{Hom} (A^{\otimes n}, M)$ for $n \geq 0$ and the coboundary map $\delta_{\mathrm{Hoch}} : C^n (A, M) \rightarrow C^{n+1}(A, M)$ is given by
\begin{align}\label{hoch-diff-for}
\delta_{\mathrm{Hoch}} (f) (a_1, \ldots, a_{n+1}) =~& a_1 f (a_2, \ldots, a_{n+1}) + \sum_{i=1}^n (-1)^i~ f (a_1, \ldots, a_{i-1}, a_i a_{i+1}, \ldots, a_{n+1}) \\
~&+ (-1)^{n+1} f(a_1, \ldots, a_n) a_{n+1}. \nonumber
\end{align}

The corresponding cohomology groups are called the Hochschild cohomology groups of $A$ with coefficients in $M$.
It has been observed by Gerstenhaber \cite{gers0} that the graded vector space of Hochschild cochains $C^\star (A, A) = \oplus_n C^n (A, A)$ carries a degree $-1$ graded Lie bracket (called the Gerstenhaber bracket) given by
\begin{align}\label{gers-brkt}
[f, g] := f \circ g - (-1)^{(m-1)(n-1)} g \circ f, ~ \text{ for } f \in C^m (A, A), ~g \in C^n (A, A),
\end{align}
where $(f \circ g) (a_1, \ldots, a_{m+n-1})$ is defined as
\begin{align*}
%(f \circ g) (a_1, \ldots, a_{m+n-1}) =
 \sum_{i=1}^m (-1)^{(i-1) (n-1)} f (a_1, \ldots, a_{i-1}, g(a_i, \ldots, a_{i+n-1}), \ldots, a_{m+n-1}).
\end{align*}

Let $\mu : A^{\otimes 2} \rightarrow A$ denote the associative multiplication on $A$. With the above notations, the Hochschild coboundary map $\delta_{\mathrm{Hoch}}$ with coefficients in itself is given by
%\begin{align*}
$\delta_{\mathrm{Hoch}} f = (-1)^{n-1} [\mu, f], ~~~~ \text{ for } f \in C^n (A, A).$
%\end{align*}

%Next, we use Hochschild cohomology for associative algebras to define cohomology of AssDer pairs. In sections \ref{sec-ext} and \ref{sec-def}, we will show that such cohomology (with appropriate coefficients) is related to central extensions of AssDer pairs and governs the one-parameter formal deformation of AssDer pairs.

Let $(A, \phi_A)$ be an AssDer pair and $(M, \phi_M)$ be a representation of it. We aim to define the cohomology of the AssDer pair  $(A, \phi_A)$ with coefficients in $(M, \phi_M)$. Define the space $C^0_{\mathrm{AssDer}}(A, M)$ of $0$-cochains to be $0$ and the space $C^1_{\mathrm{AssDer}}(A, M)$ of $1$-cochains to be $\mathrm{Hom}(A, M)$. For $n \geq 2$, the space $C^n_{\mathrm{AssDer}} (A, M)$ of $n$-cochains to be defined by
\begin{align*}
C^n_{\mathrm{AssDer}}(A, M) = C^n (A, M) \oplus C^{n-1} (A, M).
\end{align*}
Before we introduce the coboundary operator, we define a new map $\delta: C^n (A, M) \rightarrow C^n (A, M)$ by
\begin{align*}
\delta f = \sum_{i=1}^n f \circ (\mathrm{id} \otimes \cdots \otimes \phi_A \otimes \cdots \otimes \mathrm{id}) - \phi_M \circ f.
\end{align*}
When $(M, \phi_M) = (A, \phi_A)$,  the map $\delta$ can be seen in terms of Gerstenhaber bracket as
%\begin{align*}
$\delta f = - [\phi_A, f].$
%\end{align*}

Finally, we define the coboundary map $\partial : C^n_{\mathrm{AssDer}} (A, M) \rightarrow C^{n+1}_{\mathrm{AssDer}} (A, M) $ by
\begin{align}\label{ass-der-diff}
\begin{cases}
\partial f = ( \delta_{\mathrm{Hoch}} f, - \delta f ), ~~~~ \text{ for } f \in C^1_{\mathrm{AssDer}} (A, M) = \mathrm{Hom} (A, M),\\
\partial (f, g) = ( \delta_{\mathrm{Hoch}} f,~ \delta_{\mathrm{Hoch}} g + (-1)^n  \delta f), ~~~~ \text{ for } (f, g) \in C^n_{\mathrm{AssDer}} (A, M). \end{cases}
\end{align}
To prove that $\partial^2 = 0$ we use the following lemma whose proof will be given after Proposition \ref{deg-min}.

\begin{lemma}\label{commute-d-d}
 $\delta_{\mathrm{Hoch}} \circ \delta = \delta \circ \delta_{\mathrm{Hoch}}.$
\end{lemma}

\begin{prop}
The map $\partial$ is a coboundary map, i.e. $\partial^2 = 0$.
\end{prop}

\begin{proof}
For $f \in C^1_{\mathrm{AssDer}} (A, M)$, we have
\begin{align*}
\partial^2 f = \partial ( \delta_{\mathrm{Hoch}} f, - \delta f) = ( \delta_{\mathrm{Hoch}}^2 f, ~ - \delta_{\mathrm{Hoch}} \delta f + \delta \delta_{\mathrm{Hoch}} f ) = 0.
\end{align*}
For $(f, g ) \in C^n_{\mathrm{AssDer}} (A, M)$, we have
\begin{align*}
\partial^2 (f, g ) 
= ( \delta_{\mathrm{Hoch}}^2 f, ~ \delta_{\mathrm{Hoch}}^2 g + (-1)^n \delta_{\mathrm{Hoch}} \delta f + (-1)^{n+1} \delta \delta_{\mathrm{Hoch}} f ) = 0.
 \end{align*}
\end{proof}

\medskip

We denote the corresponding cohomology groups by $H^n_{\mathrm{AssDer}}(A, M)$, for $n \geq 1$.
Next, we will show that the cohomology of an AssDer pair with coefficients in itself carries a degree $-1$ graded Lie bracket.

\begin{prop}\label{deg-min}
The bracket
%\begin{align*}
$\llbracket ~, ~ \rrbracket : C^m_{\mathrm{AssDer}}(A, A) \times C^n_{\mathrm{AssDer}}(A, A) \rightarrow C^{m+n-1}_{\mathrm{AssDer}}(A, A)$ given by
%\end{align*}
\begin{align*}
\llbracket (f, g), (f', g') \rrbracket := ( [f, f'], (-1)^{m+1} [ f, g' ] + [ g. f'] )
\end{align*}
is a degree $-1$ graded Lie bracket on $\bigoplus_{n} C^n_{\mathrm{AssDer}}(A, A)$.
\end{prop}

\begin{proof}
First note that, since $[~, ~]$ is a degree $-1$ graded Lie bracket (Gerstenhaber bracket), we have
\begin{align}\label{deg-minus}
[f, [f', f'']] = [[f, f'], f'' ] + (-1)^{(m-1)(n-1)} [f', [f, f'']],
\end{align}
for $f \in \mathrm{Hom}(A^{\otimes m}, A), ~ f' \in \mathrm{Hom}(A^{\otimes n}, A)$ and $f'' \in \mathrm{Hom}(A^{\otimes p}, A)$. Next for any  $(f, g) \in C^m_{\mathrm{AssDer}}(A, A),$ $(f', g') \in C^n_{\mathrm{AssDer}}(A, A)$ and $(f'', g'') \in C^p_{\mathrm{AssDer}}(A, A),$ we have
\begin{align*}
\llbracket (f, g) , \llbracket (f', g') , (f'', g'') \rrbracket \rrbracket  &= \llbracket (f, g) , ([f', f''] , (-1)^{n+1} [f', g''] + [g', f''] ) \rrbracket \\
&= \big( [f, [f', f'']] , (-1)^{m+n+2} [f, [f', g'']] + (-1)^{m+1} [f, [g', f'' ]] + [ g, [f', f'']] \big).
\end{align*}
On the other hand,
\begin{align*}
&\llbracket \llbracket (f, g)  , (f', g') \rrbracket, (f'', g'')   \rrbracket + (-1)^{(m-1)(n-1) } \llbracket (f', g') , \llbracket (f, g) , (f'', g'') \rrbracket \rrbracket \\
&= \llbracket ([f,f'] ,~ (-1)^{m+1} [f, g'] + [ g, f']),~ (f'', g'') \rrbracket \\
&+ (-1)^{(m-1)(n-1) } \llbracket (f', g') , ~([f,f''], (-1)^{m+1} [f, g''] + [g, f'']) \rrbracket \\
&= \big( [[f,f'], f''], ~ (-1)^{m+n} [[f,f'], g''] + (-1)^{m+1} [[f, g'], f''] + [[ g, f'], f'' ]   \big)  \\
&+ (-1)^{(m-1)(n-1) } \big( [f', [f, f'']] ,~ (-1)^{n+1+m+1} [ f', [f, g'']] + (-1)^{n+1} [ f' , [g, f'']] + [g', [f, f'']]   \big). 
\end{align*}
Thus it follows from (\ref{deg-minus}) that 
\begin{align*}
\llbracket (f, g) , \llbracket (f', g') , (f'', g'') \rrbracket \rrbracket = \llbracket \llbracket (f, g)  , (f', g') \rrbracket, (f'', g'')   \rrbracket + (-1)^{(m-1)(n-1) } \llbracket (f', g') , \llbracket (f, g) , (f'', g'') \rrbracket \rrbracket,
\end{align*}
which verifies the graded Jacobi identity.
\end{proof}

It follows from the above Proposition that the shifted graded vector space $\bigoplus_n C^{n +1}_{\mathrm{AssDer}} (A, A)$ carries a graded Lie bracket. This result is true for an arbitrary vector space $A$ (not necessarily an associative algebra) and $C^{n}_{\mathrm{AssDer}} (A, A) = \mathrm{Hom}(A^{\otimes n}, A) \oplus \mathrm{Hom}(A^{\otimes n-1}, A)$, for all $n$.
Let $A$ be a vector space, $\mu : A^{\otimes 2} \rightarrow A$ and $\phi_A : A \rightarrow A$ be two linear maps. Consider the pair $(\mu, \phi_A) \in C^2_{\mathrm{AssDer}} (A, A)$. Then $\mu$ defines an associative product on $A$ and $\phi_A$ is a derivation for the associative product if and only if
$(\mu, \phi_A) \in C^2_{\mathrm{AssDer}} (A, A)$ is a Maurer-Cartan element  in the graded Lie algebra $( \bigoplus_n C^{n +1}_{\mathrm{AssDer}} (A, A), \llbracket ~, ~ \rrbracket )$, i.e.
%\begin{align*}
$\llbracket (\mu, \phi_A) , (\mu, \phi_A) \rrbracket = 0.$
%\end{align*}
With these notations, the differential (\ref{ass-der-diff}) of the AssDer pair $(A, \phi_A)$ with coefficients in itself is given by
\begin{align*}
\partial (f , g) = (-1)^{n-1} \llbracket (\mu, \phi_A), (f , g) \rrbracket, \text{ for }  (f , g) \in C^n_{\mathrm{AssDer}}(A,A).
\end{align*}
As a consequence, we get that the graded space of cohomology $\bigoplus_n H^{n+1}_{\mathrm{AssDer}} (A, A)$ of the AssDer pair $(A, \phi_A)$ with coefficients in itself carries a graded Lie bracket.

\medskip

\begin{proof} $(of~ Lemma~ \ref{commute-d-d})$ We first prove this result when the coefficient is the AssDer pair itself. Then using the semidirect product AssDer pair, we conclude the same for any coefficients.

When $(M, \phi_M ) = ( A, \phi_A)$, we have $\delta_{\mathrm{Hoch}} (f) = (-1)^{n-1} [ \mu, f]$ and $\delta f = - [ \phi_A, f]$, for any $f \in C^n(A, A)$. Hence
\begin{align*}
\delta_{\mathrm{Hoch}} \circ \delta (f) = -\delta_{\mathrm{Hoch}} [ \phi_A, f ] =~& (-1)^n [ \mu, [\phi_A, f]] \\
=~& (-1)^n \cancel{[[\mu, \phi_A], f]} + (-1)^n [\phi_A [\mu, f]] = \delta \circ \delta_{\mathrm{Hoch}} (f).
\end{align*}
For any coefficient $(M, \phi_M)$, we consider the semidirect product AssDer pair $(A \oplus M, \phi_A \oplus \phi_M)$ given in Proposition \ref{semi-d}. We use the same notation $\delta_\mathrm{Hoch}$ to denote the Hochschild cohomology of $A$ with coefficients in $M$, as well as the Hochschild cohomology of the semi-direct product algebra $A \oplus M$. Similarly, we use the same notation for the operator $\delta$.
Note that, for any $f \in C^n (A, M)$ can be extended to a map $\widetilde{f} \in C^n ( A \oplus M, A \oplus M)$ by
\begin{align*}
\widetilde{f} ((a_1, m_1), \ldots, ( a_n, m_n)) = (0, f(a_1, \ldots, a_n)).
\end{align*}
The map $f$ can be obtained from $\widetilde{f}$ just by restricting it to $A^{\otimes n}$. Moreover, $\widetilde{f} = 0$ implies that $f = 0$. Observe that $\widetilde{\delta_{\mathrm{Hoch}}(f)} = \delta_{\mathrm{Hoch}} (\widetilde{f})$ and $\widetilde{\delta f} = \delta \widetilde{f}$. Hence we have
\begin{align*}
\widetilde{\delta_{\mathrm{Hoch}} \circ \delta(f)} = \delta_{\mathrm{Hoch}} (\delta \widetilde{f}) = \delta_{\mathrm{Hoch}} \circ \delta (\widetilde{f}) = \delta \circ \delta_{\mathrm{Hoch}} ( \widetilde{f}) = \widetilde{\delta \circ \delta_{\mathrm{Hoch}} (f)}.
\end{align*}
This implies that $\delta_{\mathrm{Hoch}} \circ \delta = \delta \circ \delta_{\mathrm{Hoch}}.$
\end{proof}

\subsection{Relation with LieDer pairs}
Let $(A, \phi_A)$ be an AssDer pair. Consider the Lie algebra structure on $A$ with the commutator bracket $[a, b]_c = ab - ba.$ We denote this Lie algebra by $A_c$. Then it is easy to see that $\phi_A$ is a derivation for the Lie algebra $A_c$.
%\begin{align*}
%\phi_A [a,b]_c = \phi_A ( ab - ba ) =~& \phi_A (a) b + a \phi_A (b) - \phi_A(b) a - b \phi_A (a) \\
%=~& [\phi_A(a), b ]_c + [ a, \phi_A (b)]_c.
%\end{align*}
Thus we get a functor $(~)_c : {\bf AssDer} \rightarrow {\bf LieDer}$. In the following, we construct a functor left adjoint to $(~)_c$ using the universal enveloping algebra of a Lie algebra.

Let $(\mathfrak{g}, \phi_\mathfrak{g})$ be a LieDer pair. Consider the 
tensor algebra $T (\mathfrak{g})$ of $\mathfrak{g}$. The universal enveloping algebra $U(\mathfrak{g})$ is an associative algebra which is the quotient of $T(\mathfrak{g})$ by the two-sided ideal generated by the elements of the form $x \otimes y - y \otimes x - [x,y]$, for $x, y \in \mathfrak{g}.$
 Note that the linear map $\overline{\phi_\mathfrak{g}}  : {T} (\mathfrak{g}) \rightarrow {T} (\mathfrak{g})$ (see Example \ref{ass-der-exam} (iv)) induces a map $\phi_{U(\mathfrak{g})} : U(\mathfrak{g}) \rightarrow U(\mathfrak{g})$ as
\begin{align*}
\overline{\phi_\mathfrak{g}} ( x \otimes y - y \otimes x - [x,y] ) 
%&= \overline{\phi_\mathfrak{g}} (x) \otimes y + x \otimes \overline{\phi_\mathfrak{g}} (y) -  \overline{\phi_\mathfrak{g}} (y) \otimes x - y \otimes \overline{\phi_\mathfrak{g}} (x)   - [ \overline{\phi_\mathfrak{g}} (x), y ] - [ x, \overline{\phi_\mathfrak{g}} (y)] \\
= \overline{\phi_\mathfrak{g}} (x) \otimes y - y \otimes \overline{\phi_\mathfrak{g}} (x)  - [ \overline{\phi_\mathfrak{g}} (x), y ] + x \otimes \overline{\phi_\mathfrak{g}} (y) - \overline{\phi_\mathfrak{g}} (y) \otimes x - [ x, \overline{\phi_\mathfrak{g}} (y)].
\end{align*}
Moreover, $\phi_{U(\mathfrak{g})}$ is a derivation on $U(\mathfrak{g})$. Thus, a LieDer pair $(\mathfrak{g}, \phi_\mathfrak{g}) $ induces an AssDer pair $( U(\mathfrak{g}), \phi_{U(\mathfrak{g})} )$ on the universal enveloping algebra $U(\mathfrak{g}).$

\begin{prop}\label{left-adj-th}
The functor $U : \mathrm{\bf LieDer} \rightarrow \mathrm{\bf AssDer}$ is left adjoint to the functor $(~)_c : \mathrm{\bf AssDer} \rightarrow \mathrm{\bf LieDer}$. In other words, there is an isomorphism
\begin{align*}
\mathrm{Hom}_{\mathrm{\bf AssDer}} (U(\mathfrak{g}), A) \cong \mathrm{Hom}_{\mathrm{\bf LieDer}} (\mathfrak{g}, A_c),
\end{align*}
for any AssDer pair $(A, \phi_A)$ and any LieDer pair $(\mathfrak{g}, \phi_\mathfrak{g}).$
\end{prop}

\begin{proof}
For any AssDer pair morphism $f : U(\mathfrak{g}) \rightarrow A$, we consider its restriction to $\mathfrak{g}$, which is a Lie algebra morphism $\mathfrak{g} \rightarrow A_c$ and commute with derivations. Hence it is a morphism of LieDer pairs.
Conversely, for any LieDer pair morphism $h : \mathfrak{g} \rightarrow A_c$, we consider the unique extension of $h$ as an associative algebra morphism $\widetilde{h} : T(\mathfrak{g}) \rightarrow A$. This is indeed a morphism of AssDer pairs. It induces a map of AssDer pairs $\widetilde{\widetilde{h}} : U (\mathfrak{g}) \rightarrow A$ as $h$ is a LieDer pair morphism. Finally, the above two correspondences are inverses to each other.
\end{proof}

Let $(\mathfrak{g}, \phi_\mathfrak{g})$ be a LieDer pair. A module over it \cite{tang-f-sheng} consists of a $\mathfrak{g}$-module $M$ together with a linear map $\phi_M : M \rightarrow M$ satisfying
\begin{align*}
\phi_M [x,m] = [\phi_\mathfrak{g}(x), m ] + [ x, \phi_M (m)], ~~~ \text{ for all } x \in \mathfrak{g}, m \in M.
\end{align*}

\begin{prop}
Let $(\mathfrak{g}, \phi_\mathfrak{g})$ be a LieDer pair. A $(\mathfrak{g}, \phi_\mathfrak{g})$-module is equivalent to a left $(U(\mathfrak{g}), \phi_{U(\mathfrak{g})})$-module.
\end{prop}

\begin{proof}
It is known that any left $U(\mathfrak{g})$-module is equivalent to a $\mathfrak{g}$-module \cite{loday-book}. More precisely, let $M$ be a left $U(\mathfrak{g})$-module, then the $\mathfrak{g}$-module structure on $M$ is given by $[x, m] = xm$, for $x \in \mathfrak{g}$ and $m \in M$.

Next, take $(M, \phi_M)$ be a left module over the AssDer pair $(U(\mathfrak{g}), \phi_{U(\mathfrak{g})})$. Then the condition 
%\begin{align*}
$\phi_M (xm) = \phi_{U(\mathfrak{g})} (x)m + x \phi_M (m)$
%\end{align*}
is equivalent to $\phi_M [x,m] = [\phi_\mathfrak{g}(x), m] + [x, \phi_M (m)],$ for all $x \in \mathfrak{g}$ and $m \in M$.
\end{proof}

Let $(M, \phi_M)$ be a representation of the AssDer pair $(A, \phi_A)$. Then $M $ can be considered as an $A_c$-module via 
$[~,~]: A_c \times M \rightarrow M, ~ [a, m] = am-ma.$
 Further $(M, \phi_M)$ is a representation of the LieDer pair $(A_c, \phi_A)$.
%\begin{align*}
%\phi_M [a, m] = \phi_M (am -ma ) =~& \phi_A (a) m + a \phi_M (m) - \phi_M (m) a - m \phi_A (a) \\
%=~& \phi_A (a) m - m \phi_A (a) + a \phi_M (m) - \phi_M (m) a \\
%=~& [ \phi_A (a), m] + [a,  \phi_M (m)].
%\end{align*}
Before we relate the cohomology of an AssDer pair with that of the corresponding commutator LieDer pair, we recall the following standard result \cite{loday-book}.

\begin{prop}
The collection of maps $T_n : \mathrm{Hom}(A^{\otimes n }, M ) \rightarrow \mathrm{Hom} ( \wedge^n A_c , M), ~ n \geq 0, \mbox{defined by}$
%\begin{align*}
$$T_n (f) (a_1, \ldots, a_n ) = \sum_{\sigma \in S_n} (-1)^\sigma f (a_{\sigma (1)}, \ldots, a_{\sigma (n)})$$
%\end{align*}
is a morphism from the Hochschild complex of $A$ with coefficients in the $A$-bimodule $M$ to the Chevalley-Eilenberg cohomology of the commutator Lie algebra $A_c$ with coefficients in the module $M$.
\end{prop}
%This is standard and can be proved in a various way. First, it can be checked that $\{ T_n\}$ maps Gerstenhaber bracket to the Nijenhuis-Richardson bracket. Note that an associative structure on $A$ and bimodule $M$ can be described by a Maurer-Cartan element in Gerstenhaber Lie bracket while a Lie algebra $\mathfrak{g}$ and representation on $M$ can be described by a Maurer-Cartan element in the Nijenhuis-Richardson bracket \cite{das-rota}. The result follows as the differentials in both cases are induced by respective Maurer-Cartan elements and $T$ maps the Maurer-Cartan element for the associative structure to the Maurer-Cartan element for the corresponding Lie algebra structure.

In \cite{tang-f-sheng} the authors introduced a cohomology for a LieDer pair with coefficients in a representation. Let $(\mathfrak{g}, \phi_\mathfrak{g})$ be a LieDer pair and $(M, \phi_M)$ be a representation of it. We denote by $\delta_{\mathrm{CE}} : \mathrm{Hom}(\wedge^n \mathfrak{g}, M) \rightarrow \mathrm{Hom}(\wedge^{n+1} \mathfrak{g}, M)$ the coboundary operator for the Chevalley-Eilenberg cohomology of $\mathfrak{g}$ with coefficients in $M$. Define the $0$-th cochain group of the LieDer pair $(\mathfrak{g}, \phi_\mathfrak{g})$  with coefficients in  $(M, \phi_M)$ to be $0$, and the higher cochain groups are defined by $C^1_{\mathrm{LieDer}}(\mathfrak{g}, M) = \mathrm{Hom}(\mathfrak{g}, M)$ and $C^n_{\mathrm{LieDer}}(\mathfrak{g}, M) = \mathrm{Hom}(\wedge^n \mathfrak{g}, M) \oplus \mathrm{Hom} (\wedge^{n-1} \mathfrak{g}, M)$, for $n \geq 2$. The coboundary map $\partial : C^n_{\mathrm{LieDer}}(\mathfrak{g}, M) \rightarrow C^{n+1}_{\mathrm{LieDer}}(\mathfrak{g}, M)$ is given by
%\begin{align*}
$$\partial f = (\delta_{\mathrm{CE}}f, - \delta f) ~\mbox{and}~
\partial (f, g) = ( \delta_{\mathrm{CE}}f , ~  \delta_{\mathrm{CE}} g + (-1)^n \delta f),$$
%\end{align*}
where 
  $\delta : \mathrm{Hom}(\wedge^n \mathfrak{g}, M) \rightarrow \mathrm{Hom}(\wedge^n \mathfrak{g}, M)$ is the map
%\begin{align*}
$$\delta f = \sum_{i=1}^n f \circ ( \mathrm{id} \otimes \cdots \otimes \phi_\mathfrak{g} \otimes \cdots \otimes \mathrm{id}) - \phi_M \circ f$$
%\end{align*}
for $f \in C^1_{\mathrm{LieDer}}(\mathrm{g}, M)$ and $(f, g) \in C^n_{\mathrm{LieDer}}(\mathfrak{g}, M)$.
When one consider the cohomology of the LieDer pair $(\mathfrak{g}, \phi_\mathfrak{g})$ with coefficients in itself, the cochain groups $\bigoplus_n C^n_{\mathrm{LieDer}}(\mathfrak{g}, \mathfrak{g})$ carries a degree $-1$ graded Lie bracket $\llbracket ~, ~ \rrbracket $ given by
\begin{align*}
\llbracket (f, g), (f', g') \rrbracket = ( [f,f'] , (-1)^{m+1} [ f, g'] + [g, f']),
\end{align*}
for $(f, g) \in C^m_{\mathrm{LieDer}}(\mathfrak{g}, \mathfrak{g}), ~ (f', g') \in C^n_{\mathrm{LieDer}}(\mathfrak{g}, \mathfrak{g})$, where $[~, ~]$ is the Nijenhuis-Richardson bracket on the space of skew-symmetric multilinear maps on $\mathfrak{g}$ given by
\begin{align*}
[f, f' ] (x_1, \ldots,~& x_{m+n-1}) 
= \sum_{\sigma \in \mathrm{Sh}(n, m-1)} (-1)^\sigma f ( f' ( x_{\sigma (1)}, \ldots, x_{\sigma (n)}), x_{\sigma (n+1)}, \ldots, x_{\sigma (m+n-1)}) \\
~&- (-1)^{(m-1)(n-1)} \sum_{ \sigma \in \mathrm{Sh}(m, n-1)} (-1)^\sigma f' ( f ( x_{\sigma (1)}, \ldots, x_{\sigma (m)}), x_{\sigma (m+1)}, \ldots, x_{\sigma (m+n-1)}).
\end{align*}
If the Lie bracket on $\mathfrak{g}$ is given by a map $\omega : \mathfrak{g} \times \mathfrak{g} \rightarrow \mathfrak{g}$, then $(\omega, \phi_\mathfrak{g}) \in C^2_{\mathrm{LieDer}}(\mathfrak{g}, \mathfrak{g})$ satisfies $\llbracket (\omega, \phi_\mathfrak{g}), (\omega, \phi_\mathfrak{g}) \rrbracket = 0$ and the differential if given by
\begin{align*}
\partial (f, g) = (-1)^{n-1} \llbracket (\omega, \phi_\mathfrak{g}), (f, g) \rrbracket.
\end{align*}

\begin{prop}\label{comm-map}
Let $(A, \phi_A)$ be an AssDer pair and $(M, \phi_M)$ be a representation over it. Then the maps
\begin{align*}
(T_n, T_{n-1}) : \mathrm{Hom}(A^{\otimes n}, M) \oplus \mathrm{Hom}(A^{\otimes n-1}, M) \rightarrow \mathrm{Hom}(\wedge^n A_c, M) \oplus \mathrm{Hom}(\wedge^{n-1} A_c, M)
\end{align*}
defines a morphism from the cohomology of the AssDer pair $(A, \phi_A)$ to the cohomology of the corresponding LieDer pair $(A_c, \phi_A).$
\end{prop}

\section{Central extensions of AssDer pairs}\label{sec-ext}

%In this section, we study central extensions of an AssDer pair, i.e. extensions of an AssDer pair by a trivial AssDer pair. We show that isomorphism classes of central extensions are classified by the second cohomology of the AssDer pair with coefficients in the trivial representation.

Let $(A, \phi_A)$ be an AssDer pair and $(M, \phi_M)$ a trivial AssDer pair. That is, the associative structure on $M$ is trivial.

\begin{defn}
A central extension of $(A, \phi_A)$ by $(M, \phi_M)$ consists of an exact sequence of AssDer pairs
\begin{align}\label{c-e}
0 \rightarrow (M, \phi_M) \xrightarrow{i} (\hat{A}, \phi_{\hat{A}}) \xrightarrow{p} (A, \phi_A) \rightarrow 0
\end{align}
satisfying $i(m) \cdot \hat{a} = 0 = \hat{a} \cdot i(m)$, for all $m \in M$ and $\hat{a} \in \hat{A}$.
\end{defn}

We identify $M$ with the corresponding subalgebra of $\hat{A}$ and with this identification $\phi_M = \phi_{\hat{A}}|_M.$

\begin{defn}
Let $(\hat{A}_1, \phi_{\hat{A}_1})$ and  $(\hat{A}_2, \phi_{\hat{A}_2})$ be two central extensions of the AssDer pair $(A, \phi_A)$ by $(M, \phi_M).$ These two central extensions are said to be isomorphic if there exists an AssDer pair isomorphism $\eta : (\hat{A}_1, \phi_{\hat{A}_1}) \rightarrow (\hat{A}_2, \phi_{\hat{A}_2}) $ such that the following diagram commute
\begin{align}
\xymatrix{
 & & (\hat{A}_1, \phi_{\hat{A}_1}) \ar[dd]^{\eta}  \ar[rd]^{p_1} &  &  \\
0 \ar[r] & (M, \phi_M) \ar[rd]_{i_2} \ar[ru]^{i_1} & & (A, \phi_A) \ar[r] & 0. \\
 &  & (\hat{A}_2, \phi_{\hat{A}_2}) \ar[ru]_{p_2} &  &
}
\end{align}
\end{defn}

Let $0 \rightarrow (M, \phi_M) \xrightarrow{i} (\hat{A}, \phi_{\hat{A}}) \xrightarrow{p} (A, \phi_A) \rightarrow 0$ be a central extension of the AssDer pair $(A, \phi_A)$ by $(M, \phi_M)$. A  section of it is given by a linear map $s : A \rightarrow \hat{A}$ such that $p \circ s = \mathrm{id}_A$.

Let $s$ be a section. Define two maps $\psi : A^{\otimes 2} \rightarrow M$ and $\chi : A \rightarrow M$ by
\begin{align*}
\psi (a, b) =~ s(a) \cdot s(b) - s(ab), ~~~~ \quad \chi (a) =~ \phi_{\hat{A}} (s(a)) - s ( \phi_A (a)), ~~~ \text{ for } a, b \in A.
\end{align*}
Since the vector space $\hat{A}$ is isomorphic to $A \oplus M$ (via the section $s$), we may transfer the AssDer structure of $\hat{A}$ to $A \oplus M$. This will certainly depend on the section $s$. The product and the linear map on $A \oplus M$ are respectively given by
\begin{align*}
 (a \oplus m) \cdot (b \oplus n) =~ ab \oplus \psi (a, b) \quad \text{ and } \quad
 \phi_{A \oplus M}(a \oplus m) =~ \phi_A (a) \oplus  \phi_M (m) + \chi (a).
\end{align*}
With these notations, we have the following.

\begin{prop}\label{assder-2-co}
The pair $(A \oplus M, \phi_{A \oplus M})$ is an AssDer pair if and only if $(\psi, \chi)$ is a $2$-cocycle in the cohomology of the AssDer pair $(A, \phi_A)$ with coefficients in the trivial representation $(M= (M, l=0, r=0), \phi_M)$.
\end{prop}

%\begin{proof}
%Note that $(\psi, \chi)$ is a $2$-cocycle if and only if $\delta_{\mathrm{Hoch}} \psi = 0$ and $\delta_{\mathrm{Hoch}} \chi + \delta \psi = 0$, or equivalently,
%\begin{align}
%& \psi (ab, c ) - \psi (a, bc) = 0, \label{psi-1}\\
%& - \chi (ab) + \psi ( \phi_A a, b ) + \psi (a, \phi_A b ) - \phi_M ( \psi (a, b)) = 0. \label{psi-2}
%\end{align}
%First observe that $(A \oplus M, \phi_{A \oplus M})$ is an AssDer pair if and only if 
%\begin{align}
%((a \oplus m ) \cdot ( b \oplus n ) ) \cdot ( c \oplus p) =~& (a \oplus m ) \cdot (( b \oplus n )  \cdot ( c \oplus p) ), \label{psi-3}\\
%\phi_{A \oplus M} ( (a \oplus m) \cdot (b \oplus n)) =~& \phi_{A \oplus M} ( a \oplus m) \cdot ( b \oplus n ) + ( a \oplus m) \cdot \phi_{A \oplus M} ( b \oplus n). \label{psi-4}
%\end{align}
%The identity (\ref{psi-3}) is equivalent to (\ref{psi-1}) and the identity (\ref{psi-4}) is equivalent to (\ref{psi-2}).
%\end{proof}

\begin{thm}\label{central-fin-thm}
Let $(A, \phi_A)$ be an AssDer pair and $(M, \phi_M)$ be a trivial AssDer pair. Then the isomorphism classes of central extensions of $(A, \phi_A)$ by $(M, \phi_M)$ are classified by the second cohomology group $H^2_{\mathrm{AssDer}}(A , M)$ of the AssDer pair with coefficients in the trivial representation $(M = (M, l=0, r=0), \phi_M).$
\end{thm}

\begin{proof}
First, we show that the cohomology class of the $2$-cocycle $(\psi, \chi)$ does not depend on the choice of $s$. Let $s_1$ and $s_2$ be two sections of (\ref{c-e}). Define a map $\phi : A \rightarrow M$ by $\phi (a) = s_1(a) - s_2 (a)$. Then we have
\begin{align*}
\psi_1 (a, b) = s_1(a) \cdot s_1(b) - s_1(ab)
=~& (s_2 (a) + \phi (a)) \cdot ( s_2(b) + \phi (b) ) - s_2(ab) - \phi (ab) \\
 =~& \psi_2 (a, b) - \phi (ab) \qquad ( \mathrm{as }~~~~ \phi (a),~ \phi (b) \in M),
\end{align*}
\begin{align*}
\chi_1 (a) = \phi_{\hat{A}} (s_1(a)) - s_1 (\phi_A (a)) 
=~& \phi_{\hat{A}} (s_2 (a) + \phi (a)) - s_2 (\phi_A (a)) - \phi (\phi_A(a)) \\
%=~& \phi_{\hat{A}} (s_2 (a)) - s_2 (\phi_A (a)) + \phi_M (\phi(a)) - \phi ( \phi_A (a)) \\
=~& \chi_2 (a) + \phi_M (\phi(a)) - \phi ( \phi_A (a)).
\end{align*}
This shows that $(\psi_1, \chi_1) = (\psi_2, \chi_2) + \partial \phi$. Hence $(\psi_1, \chi_1)$ and $(\psi_2, \chi_2)$ are representative of the same cohomology class.

Let $(\hat{A}_1, \phi_{\hat{A}_1})$ and  $(\hat{A}_2, \phi_{\hat{A}_2})$ be two isomorphic central extensions of $(A, \phi_A)$ by $(M, \phi_M)$, and the isomorphism is given by $\eta$. Let $s_1: A \rightarrow \hat{A}_1$ be a section of the first central extension. Then we have
%\begin{align*}
$p_2 \circ (\eta \circ s_1) = p_1 \circ s_1 = \mathrm{id}_A.$
%\end{align*}
This shows that $s_2 := \eta \circ s_1$ is a section for the second central extension. Since $\eta$ is a morphism of AssDer pairs, we have $\eta|_M = \mathrm{id}_M$. Thus, we have
%\begin{align*}
$$\psi_2 (a, b) =~ s_2(a) \cdot s_2 (b) - s_2 (ab) 
%=~& (\eta \circ s_1) (a) \cdot (\eta \circ s_1)(b) - (\eta \circ s_1) (ab) \\
=~ \eta ( s_1(a) \cdot s_1(b) - s_1 (ab))
%=~ s_1(a)  \cdot s_1(b) - s_1 (ab) 
= \psi_1 (a, b),$$
%\end{align*}
%\begin{align*}
$$\chi_2 (a) =~ \phi_{\hat{A}_2} ( s_2 (a)) - s_2 (\phi_A (a)) 
%=~& \phi_{\hat{A}_2} (\eta \circ s_1 (a)) - (\eta \circ s_1 )(\phi_A(a)) \\
=~ \eta \big(  \phi_{\hat{A}_1} ( s_1 (a)) - s_1 (\phi_A (a)) \big)
%=~ \phi_{\hat{A}_1} ( s_1 (a)) - s_1 (\phi_A (a)) 
= \chi_1 (a).$$
%\end{align*}
This shows that isomorphic central extensions give rise to same $2$-cocycle, hence, correspond to same element in $H^2_{\mathrm{AssDer}} (A,M).$

Conversely, let $(\psi_1, \chi_1)$ and $(\psi_2, \chi_2)$ be two cohomologous $2$-cocycles. Therefore, there exists a linear map $\phi : A \rightarrow M$ such that
$(\psi_1, \chi_1) - (\psi_2, \chi_2) = \partial \phi$. Consider the corresponding AssDer pairs $(A \oplus M, \phi^1_{A \oplus M}) $ and  $(A \oplus M, \phi^2_{A \oplus M}) $ given in Proposition \ref{assder-2-co}. They are isomorphic as AssDer pairs via the map $\eta: A \oplus M \rightarrow A \oplus M$ given by
%\begin{align*}
$\eta ( a \oplus m) = a \oplus  m + \phi (a).$
%\end{align*}
In fact, $\eta$ defines an isomorphism between central extensions.
\end{proof}

\subsection{Extensions of a pair of derivations}
In this subsection, we study extensions of a pair of derivations in a central extension of associative algebras. 
%The results of the present section are analogous to the results appeared for LieDer pair \cite{tang-f-sheng}.

Let 
\begin{align}\label{c-ext}
\xymatrix{
0 \ar[r] &  M \ar[r]^i & \hat{A} \ar[r]^p & A \ar[r] & 0
}
\end{align}
be a fixed central extension of associative algebras.

\begin{defn} A pair of derivations $(\phi_A, \phi_M) \in \mathrm{Der}(A) \times \mathrm{Der} (M)$ is said to be extensible if there exists a derivation $\phi_{\hat{A}} \in \mathrm{Der} (\hat{A})$ such that
\begin{align}\label{ass-der-der}
\xymatrix{
0 \ar[r] &  (M, \phi_M) \ar[r] & (\hat{A}, \phi_{\hat{A}}) \ar[r] & (A, \phi_A) \ar[r] & 0
}
\end{align}
is an exact sequence of AssDer pairs. In other words, $(\hat{A}, \phi_{\hat{A}})$ is a central extension of $(A, \phi_A)$ by $(M, \phi_M).$
\end{defn}

Let $s : A \rightarrow \hat{A}$ be a section of the central extension (\ref{c-ext}). Define a map $\psi : A^{\otimes 2} \rightarrow M$ by
\begin{align*}
\psi (a, b) := s(a) \cdot s(b) - s(ab).
\end{align*}
For any pair of derivations $(\phi_A, \phi_M) \in \mathrm{Der}(A) \times \mathrm{Der} (M)$, we define a map $\mathrm{Ob}^{\hat{A}}_{(\phi_A, \phi_M)} : A^{\otimes 2} \rightarrow M$ by
\begin{align*}
\mathrm{Ob}^{\hat{A}}_{(\phi_A, \phi_M)} (a, b) := \phi_M ( \psi (a, b)) - \psi ( \phi_A (a), b) - \psi (a, \phi_A (b)).
\end{align*}

\begin{prop}\label{obs-class-cen}
The map $\mathrm{Ob}^{\hat{A}}_{(\phi_A, \phi_M)} : A^{\otimes 2} \rightarrow M$ is a $2$-cocycle in the Hochschild cohomology of $A$ with coefficients in the trivial representation $M = (M, l=0, r= 0).$ Moreover, the cohomology class $[ \mathrm{Ob}^{\hat{A}}_{(\phi_A, \phi_M)}] \in H^2_{\mathrm{Hoch}} (A, M)$ does not depend on the choice of sections. 
\end{prop}

\begin{proof}
Note that $\psi$ is a $2$-cocycle on $A$ with coefficients in the trivial $A$-bimodule $M$, i.e.
\begin{align}\label{eqnn-r}
\psi (ab, c) - \psi (a, bc) = 0.
\end{align}
Observe that
\begin{align*}
~&(\delta_{\mathrm{Hoch}} \mathrm{Ob}^{\hat{A}}_{(\phi_A, \phi_M)} ) (a, b, c) 
=~ - \mathrm{Ob}^{\hat{A}}_{(\phi_A, \phi_M)} (ab, c) + \mathrm{Ob}^{\hat{A}}_{(\phi_A, \phi_M)} (a, bc) \\
&= - \cancel{\phi_M ( \psi (ab, c))} + \psi ( \phi_A (ab), c) + \psi ( ab, \phi_A (c) ) + \cancel{\phi_M ( \psi (a, bc))} - \psi ( \phi_A (a), bc) - \phi ( a, \phi_A (bc))\\
&= \psi ( a \phi_A (b), c) + \psi ( \phi_A (a) b, c) + \psi (ab, \phi_A (c)) - \psi ( \phi_A(a), bc) - \psi ( a, b \phi_A(c)) - \psi ( a, \phi_A(b) c)\\
&= 0 \quad (\mathrm{by } (\ref{eqnn-r})).
\end{align*}
This proves the first part. To prove the second part, take $s_1, s_2$ to be two sections of (\ref{c-ext}). Define a map $\phi : A \rightarrow M$ by $\phi  = s_1 - s_2$. Then we get
\begin{align*}
\psi_1 (a, b) = s_1 (a) \cdot s_1(b) - s_1 (ab)
=~& (s_2 (a) + \phi (a)) \cdot (s_2 (b) + \phi(b)) - s_2(ab) - \phi (ab) \\
=~& s_2 (a) \cdot s_2 (b) - s_2 (ab) - \phi (ab) = \psi_2 (a, b) - \phi (ab).
\end{align*}
If the $2$-cocycles corresponding to $s_1$ and $s_2$ are respectively denoted by ${}^{1}\mathrm{Ob}^{\hat{A}}_{(\phi_A, \phi_M)}$ and ${}^{2}\mathrm{Ob}^{\hat{A}}_{(\phi_A, \phi_M)}$, then
\begin{align*}
&{}^{1}\mathrm{Ob}^{\hat{A}}_{(\phi_A, \phi_M)} (a, b) =~ \phi_M ( \psi_1 (a, b)) - \psi_1 ( \phi_A (a), b) - \psi_1 ( a, \phi_A (b)) \\
&= \phi_M ( \psi_2 (a, b)) - \phi_M ( \phi(ab)) - \psi_2 (\phi_A (a), b) + \phi ( \phi_A(a)b) - \psi_2 ( a, \phi_A (b)) + \phi( a \phi_A(b)) \\
&= {}^{2}\mathrm{Ob}^{\hat{A}}_{(\phi_A, \phi_M)} (a, b) + \delta_{\mathrm{Hoch}} ( \phi_M \circ \phi - \phi \circ \phi_A)(a, b).
\end{align*}
This shows that the $2$-cocycles ${}^{1}\mathrm{Ob}^{\hat{A}}_{(\phi_A, \phi_M)}$ and ${}^{2}\mathrm{Ob}^{\hat{A}}_{(\phi_A, \phi_M)}$ are cohomologous, hence, they correspond to same cohomology class in $H^2_{\mathrm{Hoch}}(A, M).$
\end{proof}

The cohomology class  $[ \mathrm{Ob}^{\hat{A}}_{(\phi_A, \phi_M)}] \in H^2_{\mathrm{Hoch}}(A, M)$ considered above is called the {\em obstruction class} to extend the pair of derivations $(\phi_A, \phi_M).$

\begin{thm}\label{thm-obs-tri}
A pair of derivations $(\phi_A, \phi_M) \in \mathrm{Der} (A) \times \mathrm{Der} (M)$ is extensible if and only if the obstruction class $[ \mathrm{Ob}^{\hat{A}}_{(\phi_A, \phi_M)}] \in H^2_{\mathrm{Hoch}} (A, M)$ is trivial.
\end{thm}

\begin{proof}
Suppose that the pair $(\phi_A, \phi_M)$ is extensible. That is, there exists a derivation $\phi_{\hat{A}} \in \mathrm{Der}( \hat{A})$ such that (\ref{ass-der-der}) is an exact sequence of AssDer pairs. Define a map $\lambda : A \rightarrow M$ by 
%\begin{align*}
$\lambda (a) = \phi_{\hat{A}} (s(a)) - s ( \phi_A (a)).$
%\end{align*}
Note that the image of $\lambda$ lies in $M$ as $p \big( \phi_{\hat{A}} (s(a)) - s ( \phi_A (a)) \big) = 0,$ which implies that $\phi_{\hat{A}} (s(a)) - s ( \phi_A (a)) \in \mathrm{ker}(p) = \mathrm{im }(i).$ 

For any $s(a) + m \in \hat{A}$, we observe that
\begin{align*}
\phi_{\hat{A}} ( s(a) + m ) = \phi_{\hat{A}} (s(a)) + \phi_M (m) 
=~& \phi_{\hat{A}} (s(a)) - s (\phi_A (a) ) + s ( \phi_A (a)) + \phi_M (m) \\
=~& s (\phi_A (a) ) + \lambda (a) + \phi_M (m).
\end{align*}
Hence, for any $s(a) +m, ~ s(b) + n \in \hat{A}$, we have
\begin{align}
\phi_{\hat{A}} ((s(a) + m) \cdot ( s(b) + n)) = \phi_{\hat{A}} ( s(a) \cdot s(b) ) \nonumber 
%=~& \phi_{\hat{A}} ( s(ab) + s(a) \cdot s(b) - s( ab)) \nonumber \\
=~&  \phi_{\hat{A}}  ( s (ab) + \psi (a, b)) \nonumber \\
=~& s ( \phi_A (ab) ) + \lambda (ab) + \phi_M ( \psi (a, b)). \label{u}
\end{align}
On the other hand,
\begin{align}
~&\phi_{\hat{A}} ( s(a) + m ) \cdot ( s(b) + n) + ( s(a) + m) \cdot \phi_{\hat{A}} ( s(b) + n) \nonumber \\
~&= (s (\phi_A (a)) + \lambda (a) + \phi_M (m) )  \cdot ( s(b) + n) + ( s(a) + m) \cdot ( s (\phi_A (b)) + \lambda (b) + \phi_M (n)   ) \nonumber \\
~&= s (\phi_A (a)) \cdot s(b) + s( a) \cdot s (\phi_A (b)) \nonumber \\
~&= s ( \phi_A (a) \cdot b ) + s ( \phi_A (a)) \cdot s(b)  - s ( \phi_A (a) \cdot b ) + s ( a \cdot \phi_A (b) ) + s(a) \cdot s( \phi_A (b)) - s( a \cdot \phi_A (b)) \nonumber \\
~&= s ( \phi_A (a) \cdot b ) + \psi ( \phi_A (a), b ) + s ( a \cdot \phi_A (b) ) +  \psi ( a, \phi_A (b)). \label{v}
\end{align}
Since $\phi_{\hat{A}}$ is a derivation, it follows from (\ref{u}) and (\ref{v}) that
\begin{align}\label{w}
\phi_M ( \psi (a, b)) - \psi ( \phi_A (a), b ) - \psi ( a, \phi_A (b) ) = - \lambda (ab ).
\end{align}
This implies that $\mathrm{Ob}^{\hat{A}}_{(\phi_A, \phi_M)} = \partial \lambda $ is given by a coboundary. Hence $[\mathrm{Ob}^{\hat{A}}_{(\phi_A, \phi_M)}]$ is trivial.

Conversely, suppose that the obstruction cocycle is given by a coboundary, say $\mathrm{Ob}^{\hat{A}}_{(\phi_A, \phi_M)} = \partial \lambda $, for some $\lambda: A \rightarrow M$. We define a map $\phi_{\hat{A}}$ on $\hat{A}$ by
\begin{align*}
\phi_{\hat{A}} ( s(a) + m) = s ( \phi_A (a)) + \lambda (a) + \phi_M (m).
\end{align*}
Using (\ref{w}), we can show that (\ref{ass-der-der}) is an exact sequence of AssDer pairs. Hence $(\phi_A, \phi_M)$ is extensible.
\end{proof}

As a consequence, we get the following.

\begin{corollary}
If $H^2_{\mathrm{Hoch}} (A, M) = 0$ then any pair of derivations $(\phi_A, \phi_M) \in \mathrm{Der} (A) \times \mathrm{Der} (M)$ is extensible.
\end{corollary}

\vspace*{0.2cm}

Let $A$ be an associative algebra and $M = (M, l=0, r=0)$ be a trivial bimodule. In the following, we give conditions on a pair of derivations $(\phi_A, \phi_M) \in \mathrm{Der} (A) \times \mathrm{Der} (M)$ such that it is extensible in every central extension of associative algebras.

Define a map $\Theta : \mathrm{Der}(A) \times \mathrm{Der} (M) \rightarrow \mathfrak{gl} (H^2_{\mathrm{Hoch}} (A, M))$ by
\begin{align*}
\Theta (\phi_A, \phi_M)([\psi]) := [ \phi_M \circ \psi - \psi \circ (\phi_A \otimes \mathrm{id} ) - \psi \circ (\mathrm{id} \otimes \phi_A)].
\end{align*}

\begin{thm}
A pair of derivations $(\phi_A, \phi_M) \in \mathrm{Der} (A) \times \mathrm{Der} (M)$ is extensible in every central extensions of $A$ by $M$ if and only if $\Theta (\phi_A, \phi_M) = 0$.
\end{thm}

\begin{proof}
 Let $0 \rightarrow M \xrightarrow{ i } \hat{A} \xrightarrow{p} A \rightarrow 0$ be any central extension of $A$ by $M$. For any section $s : A \rightarrow \hat{A}$, the map $\psi : A^{\otimes 2} \rightarrow M$, $\psi (a, b ) = s(a ) \cdot s(b) - s ( ab)$ is a $2$-cocycle in the cohomology of $A$ with coefficients in $M$. If $\Theta (\phi_A, \phi_M)  = 0$  then we have
 \begin{align*}
 [ \mathrm{Ob}^{\hat{A}}_{(\phi_A, \phi_M)} ] = [ \phi_M \circ \psi - \psi \circ ( \phi_A \otimes \mathrm{id}) - \psi \circ ( \mathrm{id} \otimes \phi_A) ] = \Theta ( \phi_A, \phi_M) ([\psi ]) = 0.
 \end{align*}
 Hence by Theorem \ref{thm-obs-tri}, the pair $(\phi_A, \phi_M)$ is extensible.
 
 Conversely, suppose that $(\phi_A, \phi_M)$ is extensible in every central extensions of $A$ by $M$. Take any class $[\psi] \in H^2_{\mathrm{Hoch}} (A, M)$. This induces a central extension of $A$ by $M$:
 \begin{align}\label{eqn-m}
 0 \rightarrow M \xrightarrow{i} A \oplus M \xrightarrow{p} A \rightarrow 0,
 \end{align}
 where the associative product on $A \oplus M$ is given by
% \begin{align*}
$ ( a \oplus m ) \cdot ( b \oplus n) = ab \oplus \psi (a, b).$
% \end{align*}
 Since $( \phi_A, \phi_M)$ is extensible in the central extension (\ref{eqn-m}), by Theorem \ref{thm-obs-tri} we have
 \begin{align*}
 \Theta ( \phi_A, \phi_M) ( [ \psi ]) = [ \phi_M \circ \psi - \psi \circ ( \phi_A \otimes \mathrm{id}) - \psi \circ ( \mathrm{id} \otimes \phi_A) ] = [ \mathrm{Ob}^{A \oplus M}_{(\phi_A, \phi_M)} ] = 0.
 \end{align*}
 This shows that $\Theta ( \phi_A, \phi_M) = 0$.
\end{proof}

\section{Deformations}\label{sec-def}
%In this section, we study formal deformations of an AssDer pair by considering deformations of both the algebra multiplication and the derivation. Our main results are similar to the classical cases \cite{gers,nij-ric}.

Let $(A, \phi_A)$ be an AssDer pair. We denote the associative multiplication on $A$ by $\mu$. Consider the space $A[[t]]$ of formal power series in $t$ with coefficients from $A$. Then $A[[t]]$ is a $\mathbb{K}[[t]]$-module.
% and when $A$ is finite dimensional, we have $A[[t]] \cong A \otimes_\mathbb{K} \mathbb{K}[[t]].$

A formal (one-parameter) deformation of the AssDer pair $(A, \phi_A)$ consists of two formal power series
\begin{align*}
\mu_t =~& \sum_{i \geq 0} t^i \mu_i, ~~~~~ \mu_i \in \mathrm{Hom} (A^{\otimes 2}, A) \text{ with } \mu_0 = \mu,\\
\phi_t =~& \sum_{i \geq 0} t^i \phi_i, ~~~~~ \phi_i \in \mathrm{Hom} (A, A) ~~ \text{ with } \phi_0 = \phi_A
\end{align*}
such that the $\mathbb{K}[[t]]$-module $A[[t]]$ together with the multiplication $\mu_t$ forms an associative algebra and $\phi_t : A[[t]] \rightarrow A[[t]]$ is a derivation on it. In other words, $A[[t]]$ with the associative multiplication $\mu_t$ and the derivation $\phi_t$ forms an AssDer pair over $\mathbb{K}[[t]]$.
It is clear from the definition that $\mu_t  = \sum_{i \geq 0} t^i \mu_i$ defines a deformation of the associative structure on $A$ in the sense of Gerstenhaber \cite{gers}.

Let $(\mu_t, \phi_t)$ defines a deformation of the AssDer pair $(A, \phi_A)$. Then we have
\begin{align*}
\mu_t ( \mu_t (a, b), c) =~  \mu_t (a, \mu_t (b, c) )  \quad \text{ and } \quad 
\phi_t ( \mu_t (a, b)) =~ \mu_t ( \phi_t (a), b) + \mu_t ( a, \phi_t (b)),
\end{align*}
for all $a, b, c \in A$.
Expanding both the equations as power series in $t$ and equating coefficients of $t^n$ in both the equations, we get for $n \geq 0$,
\begin{align}
\sum_{i+j = n} \mu_i (\mu_j (a, b), c) =~& \sum_{i+j = n} \mu_i ( a, \mu_j (b, c)), \label{def-eq-1}\\
\sum_{i+j = n} \phi_i (\mu_j (a, b)) =~& \sum_{i+j = n} \mu_i (\phi_j (a), b) + \mu_i (a, \phi_j (b)). \label{def-eqn-2}
\end{align}
For $n = 0$, the identity (\ref{def-eq-1}) and (\ref{def-eqn-2}) both holds automatically. However, for $n =1$, we obtain
\begin{align}
\mu_1 (ab, c) + \mu_1 (a, b) c =~& a \mu_1 (b, c) + \mu_1 (a, bc) , \label{def-n-1}\\
\phi (\mu_1 (a, b)) + \phi_1 (ab) =~&  \phi_1 (a) b + \mu_1 ( \phi(a), b) + a \phi_1 (b) + \mu_1 (a, \phi(b)). \label{def-n-2}
\end{align}
The identity (\ref{def-n-1}) is equivalent $\delta_{\mathrm{Hoch}}(\mu_1) = 0$ while the identity  (\ref{def-n-2}) is equivalent to $\delta_{\mathrm{Hoch}}( \phi_1) + \delta \mu_1 = 0$. It follows from (\ref{ass-der-diff}) that $\partial ( \mu_1, \phi_1) = 0.$ Hence we get the following.

\begin{prop}\label{inf-2}
Let $(\mu_t, \phi_t)$ be a formal deformation of an AssDer pair $(A, \phi_A)$. Then the linear term $(\mu_1, \phi_1)$ is a $2$-cocycle in the cohomology of the AssDer pair $(A, \phi_A)$ with coefficients in itself.
\end{prop}

The $2$-cocycle $(\mu_1, \phi_1)$ is called the infinitesimal of the formal deformation $(\mu_t, \phi_t).$ 

\begin{defn}
Let $(\mu_t, \phi_t)$ and $(\mu_t', \phi_t')$ be two formal deformations of an AssDer pair $(A, \phi_A)$. They are said to be equivalent if there exists a formal isomorphism $\Phi_t = \sum_{i \geq 0} t^i \Phi_i: A[[t]] \rightarrow A[[t]]$ with $\Phi_0 = \mathrm{id}_A$, such that
\begin{align*}
\Phi_t \circ \mu_t =~ \mu_t' \circ ( \Phi_t \otimes \Phi_t )  \quad \text{ and } \quad 
\Phi_t \circ \phi_t =~ \phi_t' \circ \Phi_t.
\end{align*}
\end{defn}

It follows that the following identities must hold (by equating coefficients of $t^n$ from both sides)
\begin{align*}
\sum_{i+j = n} \Phi_i \circ \mu_j =~ \sum_{i +j +k = n } \mu_i' \circ (\Phi_j \otimes \Phi_k) \quad \text{ and } \quad 
\sum_{i+j = n} \Phi_i \circ \phi_j =~ \sum_{i+j = n } \phi_i' \circ \Phi_j.
\end{align*}
For $n = 0$, both the identities hold as $\Phi_0 = \mathrm{id}_A$. For $n = 1$, we obtain
\begin{align*}
\mu_1 + \Phi_1 \circ \mu =~& \mu_1' + \mu \circ ( \Phi_1 \otimes \mathrm{id} ) + \mu \circ (\mathrm{id} \otimes \Phi_1 ) ~~~~ \text{ and } ~~~~
\phi_1 + \Phi_1 \circ \phi_A = \phi_1' + \phi_A \circ \Phi.
\end{align*}
This implies that $(\mu_1, \phi_1) - (\mu_1' , \phi_1') = \partial (\Phi_1).$
Thus we have the following.

\begin{thm}
The infinitesimals corresponding to equivalent deformations of an AssDer pair $(A, \phi_A)$ are cohomologous. Therefore, they correspond to the same cohomology class.
\end{thm}

To obtain a one-to-one correspondence between the second cohomology group $H^2_{\mathrm{AssDer}}(A, A)$ and equivalence classes of certain type deformations, we use the truncated version of formal deformations.

\begin{defn}
An infinitesimal deformation of an AssDer pair $(A, \phi_A)$ is a deformation of $(A, \phi_A)$ over $\mathbb{K}[[t]]/ (t^2)~~$ ~~(the local Artinian ring of dual numbers).
\end{defn}

Thus, an infinitesimal deformation of $(A, \phi_A)$ consists  of a pair $(\mu_t, \phi_t)$ in which $\mu_t = \mu + t \mu_1$ and $\phi_t = \phi_A + t \phi_1$ such that $(\mu_1, \phi_1)$ is a $2$-cocycle in the cohomology of the AssDer pair $(A, \phi_A)$. 
%With this definition, we have the following.

\begin{thm}
There is a one-to-one correspondence between the space of equivalence classes of infinitesimal deformations of the AssDer pair $(A, \phi_A)$ and the second cohomology group $H^2_{\mathrm{AssDer}}(A, A).$
\end{thm}

\begin{proof}
It is already shown that the map
\begin{align*}
\mathrm{infinitesimal~ deformations} / \sim ~~ \longrightarrow ~~ H^2_{\mathrm{AssDer}}(A, A) ~\text{ given by }~  [(\mu_t, \phi_t)] \mapsto [(\mu_1, \phi_1)]
\end{align*}
is well-defined. This map is bijective with the inverse map given as follows. For any $2$-cocycle  $(\mu_1, \phi_1) \in C^2_{\mathrm{AssDer}}(A, A)$, the pair $(\mu_t = \mu + t \mu_1, \phi_t = \phi_A + t \phi_1)$ defines an infinitesimal deformation of $(A, \phi_A)$. If $(\mu_1', \phi_1') \in C^2_{\mathrm{AssDer}}(A, A)$ is another $2$-cocycle cohomologous to $(\mu_1, \phi_1)$, then we have $(\mu_1, \phi_1) - (\mu_1' , \phi_1') = \partial h,$ for some $h \in \mathrm{Hom}(A, A) = C^1_{\mathrm{AssDer}}(A, A)$. In such a case, $\Phi_t = \mathrm{id}_A + t h$ defines an equivalence between the infinitesimal deformations $(\mu_t = \mu + t \mu_1, \phi_t = \phi_A + t \phi_1)$ and 
$(\mu_t' = \mu + t \mu_1', \phi_t' = \phi_A + t \phi_1')$. Therefore, the inverse map is also well-defined.
\end{proof}

\begin{defn}
A formal deformation $(\mu_t, \phi_t)$ of an $\mathrm{AssDer}$ pair $(A, \phi_A)$ is said to be trivial if it is equivalent to $(\mu_t' = \mu, \phi_t' = \phi_A)$.
\end{defn}

\begin{thm}\label{h-two}
If $H^2_{\mathrm{AssDer}} (A, A) = 0$ then every formal deformation of the AssDer pair $(A, \phi_A)$ is trivial.
\end{thm}

\begin{proof}
Let $(\mu_t, \phi_t)$ be any formal one-parameter deformation of $(A, \phi_A)$. It follows from Proposition \ref{inf-2} that the linear term $(\mu_1, \phi_1)$ is a $2$-cocycle. From the given hypothesis, there exists a $1$-cochain $\Phi_1 \in C^1_{\mathrm{AssDer}} (A, A) = \mathrm{Hom}(A, A)$ such that
%\begin{align*}
$(\mu_1, \phi_1) = \partial \Phi_1.$
%\end{align*}
Setting $\Phi_t = \text{id}_A + \Phi_1 t : A[[t]] \rightarrow A[[t]]$ and define
\begin{align}\label{h2=0}
\mu_t' = \Phi_t^{-1} \circ \mu_t \circ (\Phi_t \otimes \Phi_t), \qquad \phi_t' = \Phi_t^{-1} \circ \phi_t \circ \Phi_t.
\end{align}
Then $(\mu_t', \phi_t')$ is equivalent to $(\mu_t, \phi_t)$. Moreover, it follows from (\ref{h2=0}) that $\mu_t'$ and $\phi_t'$ are of the form
$\mu_t' = \mu + t^2 \mu_2'  + \cdots  ~~~ \mbox{and} ~~~ \phi_t' = \phi_A + t^2 \phi_2' + \cdots.$
In other words, the linear terms of $\mu_t'$ and $\phi_t'$ vanish.  By repeating this argument, one can show that $(\mu_t, \phi_t)$ is equivalent to $(\mu, \phi_A).$
\end{proof}

\begin{remark}\label{h-two-rem}
An AssDer pair $(A, \phi_A)$ is said to be rigid if every formal deformation is equivalent to $(\mu, \phi_A)$. It follows that the vanishing of the second cohomology is a sufficient condition for the rigidity.
\end{remark}

\subsection{Extensions of finite order deformation}
%In this subsection, we consider deformations of order $n$ of an AssDer pair. To any deformation of order $n$, we associate a $3$-cocycle in the cohomology of the AssDer pair with coefficients in itself. We show that such a deformation extends to deformation of order $(n+1)$ if and only if the corresponding cohomology class is trivial.

Let $(A, \phi_A)$ be an AssDer pair. Consider the $\mathbb{K}[[t]]/ (t^{n+1})$-module $A[[t]]/ (t^{n+1})$. A deformation of order $n$ of the AssDer $(A, \phi_A)$ consists of a pair $(\mu_t, \phi_t )$ where $\mu_t = \sum_{i=0}^n t^i \mu_i$ and $\phi_t = \sum_{i=0}^n t^i \phi_i$ such that $\mu_t$ defines an associative product on $A[[t]]/ (t^{n+1})$ and $\phi_t$ defines a derivation on it.

Thus, in a deformation of order $n$, the following identities must hold
\begin{align*}
\sum_{i+j = k} \mu_i (\mu_j (a, b), c) = \sum_{i+j = k} \mu_i ( a, \mu_j (b, c)) ~~~ \text{ and } ~~~
\sum_{i+j = k} \phi_i (\mu_j (a, b)) = \sum_{i+j = k} \mu_i (\phi_j (a), b) + \mu_i (a, \phi_j (b)),
\end{align*}
for $k = 0, 1, \ldots, n.$ In other words,
\begin{align}
\delta_\mathrm{Hoch} (\mu_k) =~& \frac{1}{2} \sum_{i+j=k, i, j > 0} [\mu_i, \mu_j], \quad \mbox{and}\label{n-def-eqn-1}\\
\delta_\mathrm{Hoch} (\phi_k) + \delta (\mu_{k}) =~& \sum_{i+j = k, i, j >0} [\phi_i, \mu_j ].  \label{n-def-eqn-2}
\end{align}
Let $(\mu_{n+1}, \phi_{n+1}) \in C^2_{\mathrm{AssDer}} (A, A)$ be such that $(\mu_t' = \sum_{i=0}^n t^i \mu_i + t^{n+1} \mu_{n+1} ,~ \phi_t' = \sum_{i=0}^n t^i \phi_i + t^{n+1} \phi_{n+1})$ defines a deformation of order $n+1$.
Then the deformation $(\mu_t = \sum_{i=0}^n t^i \mu_i, \phi_t = \sum_{i=0}^n t^i \phi_i)$ is said to be extensible.
In such a case, two more equations need to be satisfied, namely,
%\begin{align}
%\sum_{i+j = n+1} \mu_i (\mu_j (a, b), c) =~& \sum_{i+j = n+1} \mu_i ( a, \mu_j (b, c)) \quad \mbox{and} \label{n-plus-one-1}\\
%\sum_{i+j = n+1} \phi_i (\mu_j (a, b)) =~& \sum_{i+j = n+1} \mu_i (\phi_j (a), b) + \mu_i (a, \phi_j (b)). \label{n-plus-one-2}
%\end{align}
%The identities (\ref{n-plus-one-1}) and (\ref{n-plus-one-2}) can be written as 
%\begin{align*}
\begin{equation*}
\begin{split}
\delta_{\mathrm{Hoch}}(\mu_{n+1})  =~& \frac{1}{2} \sum_{i+j = n+1, i, j > 0} [\mu_i, \mu_j] \quad (= \mathrm{Ob}^3(a,b, c) ~ \mathrm{say}),\\
\delta_{\mathrm{Hoch}} (\phi_{n+1}) + \delta (\mu_{n+1}) =~&  \sum_{i+j = n+1, i, j > 0} [\phi_i, \mu_j]  \quad (= \mathrm{Ob}^2(a, b) ~ \mathrm{say}). 
\end{split}
\end{equation*}
%Note that, in terms of the Gerstenhaber bracket (\ref{gers-brkt}), we have
%\begin{align*}
%\mathrm{Ob}^3 =  ~~~ \qquad \mathrm{ and }\qquad \mathrm{Ob}^2 = .
%\end{align*}

\begin{prop}\label{obs-class}
The pair $(\mathrm{Ob}^3, \mathrm{Ob}^2)$ is a $3$-cocycle in the cohomology of the AssDer pair $(A, \phi_A)$ with coefficients in itself.
\end{prop}

\begin{proof}
It is known from the finite order deformations of associative algebras \cite{gers} that the obstruction $\mathrm{Ob}^3$ is a Hochschild $3$-cocycle in the cohomology of $A$, i.e. $\delta_{\mathrm{Hoch}} ( \mathrm{Ob}^3) = 0$. Moreover, we have
\begin{equation*}
\begin{split}
&\delta_{\mathrm{Hoch}} (\mathrm{Ob}^2) + (-1)^3~ \delta (\mathrm{Ob}^3) 
= - [\mu, \mathrm{Ob}^2] + [\phi_A, \mathrm{Ob}^3] \\
%&= - \sum_{i+j = n+1, i, j > 0} [\mu, [\phi_i, \mu_j ]] + \frac{1}{2} \sum_{ \substack{i+j = n+1 \\ i, j > 0 }} [\phi_A, [\mu_i, \mu_j]] \\
&= - \sum_{ \substack{i+j = n+1 \\ i, j > 0 }} \big(  [[\mu, \phi_i], \mu_j] + [\phi_i, [\mu, \mu_j]]    \big) + \frac{1}{2} \sum_{ \substack{i+j = n+1 \\ i, j > 0 }} \big(  [[\phi_A, \mu_i], \mu_j] + [\mu_i, [\phi_A, \mu_j]]   \big) \\
&= - \sum_{ \substack{i+j = n+1 \\ i, j > 0 }} \big(  [[\mu, \phi_i], \mu_j] + [\phi_i, [\mu, \mu_j]]    \big) + \sum_{ \substack{i+j = n+1 \\ i, j > 0 }} [[\phi_A, \mu_i], \mu_j] \\
&= - \sum_{ \substack{i+j = n+1 \\ i, j > 0 }}
\big(  [[\mu, \phi_i], \mu_j]      - [[\phi_A, \mu_i], \mu_j]          ) + \frac{1}{2} 
\sum_{ \substack{i+j'+j'' = n+1 \\ i, j', j'' > 0 }}
[\phi_i, [\mu_j', \mu_j'']] 
%+ \sum_{\begin{array}{c} {i+j = n+1},\\{ i, j > 0}\end{array}} [[\phi_A, \mu_i], \mu_j] \quad
 (\mathrm{by }~~ (\ref{n-def-eqn-1}))\\
&= - \sum_{ \substack{i'++i''+j = n+1 \\ i', i'', j > 0 }}
 [[\phi_{i'}, \mu_{i''} ], \mu_j ] - \sum_{ \substack{i+j = n+1 \\ i, j > 0 }} ( [[\phi_A, \mu_i], \mu_j] - [[ \phi_A, \mu_i] , \mu_j] ) \\
& \quad + \frac{1}{2} \sum_{ \substack{i+j'+j'' = n+1 \\ i, j', j'' > 0 }}
 ( [[\phi_i, \mu_{j'}], \mu_{j''}] +[ \mu_{j'}, [ \phi_i, \mu_{j''} ]] )
%+ \sum_{i+j = n+1, i, j > 0} [[\phi_A, \mu_i], \mu_j] \quad 
(\mathrm{by }~~ (\ref{n-def-eqn-2}))\\
&= - \sum_{ \substack{i' + i''+j = n+1 \\ i', i'', j > 0 }}
[[\phi_{i'}, \mu_{i''}], \mu_j] + \sum_{ \substack{i+j' + j'' = n+1 \\ i, j', j'' > 0 }}
 [[ \phi_i, \mu_{j'}], \mu_{j''}] = 0.
\end{split}
\end{equation*}
Thus,
%\begin{align*}
$\partial (\mathrm{Ob}^3, \mathrm{Ob}^2) = (\delta_{\mathrm{Hoch}} (\mathrm{Ob}^3), \delta_{\mathrm{Hoch}}( \mathrm{Ob}^2) + (-1)^3~ \delta (\mathrm{Ob}^3)) = 0.$
%\end{align*}
\end{proof}

Therefore, $(\mathrm{Ob}^3, \mathrm{Ob}^2)$ defines a cohomology class in $H^3_{\mathrm{AssDer}}(A, A)$. If this cohomology class vanishes, i.e.  $(\mathrm{Ob}^3, \mathrm{Ob}^2)$ is a coboundary, then we have
%\begin{align*}
$\partial (\mu_{n+1}, \phi_{n+1}) = (\mathrm{Ob}^3, \mathrm{Ob}^2 ),$
%\end{align*}
for some $(\mu_{n+1}, \phi_{n+1}) \in C^2_{\mathrm{AssDer}}(A, A)$. In such a case $(\mu_t' = \mu_t + t^{n+1} \mu_{n+1} , \phi_t' = \phi_t + t^{n+1} \phi_{n+1} )$ defines a deformation of order $n+1$. Therefore, the deformation $(\mu_t, \phi_t)$ becomes extensible. On the other hand, if $(\mu_t, \phi_t)$ is extensible, there exists $(\mu_{n+1}, \phi_{n+1}) \in C^2_{\mathrm{AssDer}}(A, A)$ such that $(\mu_t' = \mu_t + t^{n+1} \mu_{n+1} , \phi_t' = \phi_t +  t^{n+1} \phi_{n+1})$ is a deformation of order $n+1$. Hence the obstruction $(\mathrm{Ob}^3, \mathrm{Ob}^2)$ is given by the coboundary $\partial (\mu_{n+1}, \phi_{n+1})$. Thus the corresponding cohomology class is null. Therefore, we obtain the following.

\begin{thm}\label{obs-class-zero}
Let $(\mu_t, \phi_t)$ be a deformation of order $n$ of the AssDer pair $(A, \phi_A)$. It is extensible if and only if the obstruction class $[(\mathrm{Ob}^3, \mathrm{Ob}^2)]$ vanishes.
\end{thm}

\begin{thm}
If $H^3_{\mathrm{AssDer}} (A, A) = 0$ then every finite order deformation of the AssDer pair $(A, \phi_A)$ extends to a deformation of the next order.
\end{thm}

\begin{corollary}
If $H^3_{\mathrm{AssDer}} (A, A) = 0$ then every $2$-cocycle is the infinitesimal of a formal deformation of $(A, \phi_A)$.
\end{corollary}

\subsection{Automorphisms of the deformed AssDer pair}

Let $(A, \phi_A)$ be an AssDer pair and $(\mu_t, \phi_t)$ be a deformation of it. Suppose $\Phi_t = \sum_{i \geq 0} t^i \Phi_i : A [[t]] \rightarrow A[[t]]$ is an automorphism of the deformed AssDer pair $(A[[t]], \mu_t, \phi_t)$. Then we have
\begin{align*}
\Phi_t \circ \mu_t = \mu_t \circ (\Phi_t \otimes \Phi_t)   ~~~~ \text{ and } ~~~~ \Phi_t \circ \phi_t = \phi_t \circ \Phi_t.
\end{align*}
This in particular implies that
\begin{align*}
\Phi_1 (ab) = \Phi_1 (a) b + a \Phi_1 (b) ~~~~ \text{ and } ~~~~ \Phi_1 \circ \phi_A = \phi_A \circ \Phi_1.
\end{align*}
Therefore, the linear term $\Phi_1$ of the automorphism $\Phi_t$ is a derivation on $A$ commuting with $\phi_A$. Thus, one may now ask when a derivation on $A$ which commutes with $\phi_A$ can be extended to an automorphism of the deformed AssDer pair $(A[[t]], \mu_t, \phi_t).$ We will consider a more general situation about extensions of a finite order automorphism of the deformed AssDer pair.

Let $\Phi_t = \sum_{i=1}^N t^i \Phi_i$ be an automorphism of order $N$. It is said to be extensible if there exists a map $\Phi_{N+1}: A \rightarrow A$ such that $\Phi_t' = \sum_{i=1}^{N+1} t^i \Phi_i$ is an automorphism of order $N+1$. In other words, the following additional identities must hold:
\begin{align}
a \Phi_{N+1}(b) + \Phi_{N+1}(a) b - \Phi_{N+1} (ab) =~&   \sum_{i + j = N+1, i \neq N+1} \Phi ( \mu_j (a, b)) - \sum_{i+j+k =N+1, j, k \neq N+1} \mu_i ( \Phi_j (a), \Phi_k (b)),  \label{new-der-ob1} \\
& \quad (= \mathrm{Ob}^1_{\Phi_t} \text{ say}) \nonumber \\
- \Phi_{N+1} \circ \phi_A + \phi_A \circ \Phi_{N+1} =~& \sum_{i + j = N+1, i \neq N+1} \Phi_i \circ \phi_j - \sum_{i+j = N+1, j \neq N+1} \phi_i \circ \Phi_j \label{new-der-ob2}\\
& \quad (= \mathrm{Ob}^2_{\Phi_t} \text{ say}). \nonumber
\end{align}
The pair $(\mathrm{Ob}^1_{\Phi_t}, \mathrm{Ob}^2_{\mathrm{\Phi_t}})$ is called the obstruction to extending the order $N$ automorphism $\Phi_t$. It has been shown in \cite{fox} that $\mathrm{Ob}^1_{\Phi_t}$ is a Hochschild $2$-cocycle. It is also not difficult to show that (similar to Proposition \ref{obs-class})
\begin{align*}
\delta_{\mathrm{Hoch}} ( \mathrm{Ob}^2_{\mathrm{\Phi_t}} ) + \delta ( \mathrm{Ob}^1_{\mathrm{\Phi_t}} ) = 0.
\end{align*}
In other words, $(\mathrm{Ob}^1_{\Phi_t}, \mathrm{Ob}^2_{\Phi_t})$ is a $2$-cocycle in the cohomology of the AssDer pair $(A, \phi_A)$.

Hence from (\ref{new-der-ob1}) and (\ref{new-der-ob2}), we get the following.

\begin{thm}\label{auto-theorem}
An order $N$ automorphism $\Phi_t = \sum_{i=1}^N t^i \Phi_i$ of the deformed AssDer pair is extensible if and only if the obstruction class $[(\mathrm{Ob}^1_{\Phi_t}, \mathrm{Ob}^2_{\mathrm{\Phi_t}})] \in H^2_{\mathrm{AssDer}} (A, A) $ vanishes.
\end{thm}

\section{Homotopy derivations on $2$-term $A_\infty$-algebras}\label{sec-hd}
In this section, we are interested in $2$-term $A_\infty$-algebras \cite{stasheff} with homotopy derivations. Note that homotopy derivation on $A_\infty$-algebras was studied by Loday \cite{loday} and further developed by Doubek-Lada \cite{doubek-lada}. 
%However, we will be most interested in $A_\infty$-algebras whose underlying graded vector space is concentrated in degrees $0$ and $1$. 
We classify homotopy derivations on skeletal and strict $A_\infty$-algebras.

\begin{defn}\label{defn-2-ass-term}
A $2$-term $A_\infty$-algebra consists of a chain complex $A:= ( A_1 \xrightarrow{d} A_0 )$ together with maps $\mu_2 : A_i \otimes A_j \rightarrow A_{i+j}$, for $0 \leq i, j, i+j \leq 1$ and a map $\mu_3 : A_0 \otimes A_0 \otimes A_0 \rightarrow A_1$ satisfying the followings: for any $a, b, c, e \in A_0$ and $m, n \in A_1$,

\begin{itemize}
        \item[(a)] $d \mu_2 (a,m) = \mu_2 (a, dm)$,
        \item[(b)] $d \mu_2 (m,a) = \mu_2 (dm, a)$,
        \item[(c)] $\mu_2 (dm, n) = \mu_2 (m, dn)$, 
        \item[(d)] $d\mu_3 (a, b, c) = \mu_2 \big(   \mu_2 (a,b), c \big) - \mu_2 \big(   a, \mu_2 (b,c) \big)$,
        \item[(e1)] $\mu_3 (a, b, dm) =  \mu_2 \big(  \mu_2 (a,b), m \big) - \mu_2 \big(  a, \mu_2 (b,m) \big),$
        \item[(e2)]     $\mu_3 (a, dm, c) =  \mu_2 \big(  \mu_2 (a,m), c \big) - \mu_2 \big(  a, \mu_2 (m,c) \big) $,
        \item[(e3)]  $\mu_3 (dm, b, c) =  \mu_2 \big(  \mu_2 (m,b), c \big) - \mu_2 \big(  m, \mu_2 (b,c) \big) $,
        \item[(f)] $\mu_3  \big( \mu_2 (a,b), c, e \big) - \mu_3 \big(  a, \mu_2 (b,c), e \big) + \mu_3 \big(   a, b, \mu_2 (c,e) \big) 
         = \mu_2 \big(   \mu_3 (a,b,c), e \big) + \mu_2 \big(  a, \mu_3 (b,c,e) \big).$ 
    \end{itemize}
\end{defn}

A $2$-term $A_\infty$-algebra as above may be denoted by $(A_1 \xrightarrow{d} A_0, \mu_2, \mu_3)$. 
%When $A_1 = 0$, one simply get an associative algebra structure on $A_0$ with the multiplication given by $\mu_2 : A_0 \otimes A_0 \rightarrow A_0.$
%A $2$-term $A_\infty$-algebra $(A_1 \xrightarrow{d} A_0, \mu_2, \mu_3)$ is said to be skeletal if the differential $d = 0$. Skeletal algebras are related to Hochschild $3$-cocycles of associative algebras. There is a one-to-one correspondence between skeletal algebras and triples $(A, M, \theta),$ where $A$ is an associative algebra, $M$ is an $A$-bimodule and $\theta$ is a Hochschild $3$-cocycle of $A$ with coefficients in $M$ \cite{das}. More precisely, let $( A_1 \xrightarrow{0} A_0, \mu_2, \mu_3)$ be a skeletal algebra. Then $(A_0, \mu_2)$ is an associative algebra, $A_1$ is an $A_0$-bimodule by $l (a, m) = \mu_2 (a, m)$ and $r (m, a) = \mu_2 (m, a)$; the map $\mu_3 : A_0 \otimes A_0 \otimes A_0 \rightarrow A_1$ defines a $3$-cocycle of $A_0$ with coefficients in $A_1$.

\begin{defn}
Let $(A_1 \xrightarrow{d} A_0, \mu_2, \mu_3)$ and $(A_1' \xrightarrow{d'} A_0', \mu_2', \mu_3')$ be $2$-term $A_\infty$-algebras. A morphism between them consists of
 a chain map $f : A \rightarrow A'$ (which consists of linear maps $f_0 : A_0 \rightarrow A_0'$ and $f_1 : A_1 \rightarrow A_1'$ with $f_0 \circ d = d' \circ f_1$)
and a bilinear map $f_2 : A_0 \otimes A_0 \rightarrow A_1'$
such that for any $a, b, c \in A_0$ and $m \in A_1$, the following conditions hold
\begin{itemize}
\item[(a)] $d' f_2 (a,b) = f_0 (\mu_2 (a,b)) - \mu_2' (f_0 (a), f_0 (b)),$
\item[(b)] $f_2 (a, dm) = f_1 (\mu_2 (a,m)) - \mu_2' (f_0 (a), f_1(m)),$
\item[(c)] $f_2 (dm, a) = f_1 (\mu_2 (m, a)) - \mu_2' (f_1 (m), f_0 (a))$,
\item[(d)] $f_2 (\mu_2 (a,b), c) - f_2 ( a, \mu_2 (b, c)) - \mu_2' (f_2 (a,b), f_0 (c)) + \mu_2' ( f_0 (a), f_2 (b,c)) \\
= f_1 (\mu_3 (a,b, c)) - \mu_3' (f_0 (a), f_0 (b), f_0 (c)).$
\end{itemize}
\end{defn}

We denote the category of $2$-term $A_\infty$-algebras and morphisms between them by $2\mathrm{\bf A}_\infty$.

\begin{defn}\label{ass-der-inf-def}
Let $(A_1 \xrightarrow{d} A_0, \mu_2, \mu_3)$ be a $2$-term $A_\infty$-algebra. A homotopy derivation of degree $0$ on it consists of a chain map $\theta : A \rightarrow A$ (which consists of linear maps $\theta_i : A_i \rightarrow A_i$, for $i = 0 ,1$ satisfying $\theta_0 \circ d = d \circ \theta_1$) and $\theta_2 : A_0 \otimes A_0 \rightarrow A_1$ satisfying the followings: for any $a, b, c \in A_0$ and $m \in A_1$,
\begin{itemize}
\item[(a)] $d ( \theta_2 (a, b)) = \mu_2 (\theta_0 a, b) + \mu_2 (a, \theta_0 b) - \theta_0 (\mu_2 (a, b))$,
\item[(b)] $\theta_2 (a, dm) = \mu_2 (\theta_0 a, m) + \mu_2 (a, \theta_1 m)  - \theta_1 (\mu_2 (a, m)),$
\item[(c)] $\theta_2 (dm, a) = \mu_2 ( \theta_1 m, a) + \mu_2 (m, \theta_0 a) - \theta_1 (\mu_2 (m, a)),$
\item[(d)] $\theta_1 ( \mu_3 (a, b, c)) = \theta_2 (a, \mu_2 (b, c)) - \theta_2 (\mu_2 (a, b), c) + \mu_2 (a, \theta_2 (b, c)) - \mu_2 (\theta_2 (a, b), c) \\+ \mu_3 (\theta_0 a, b, c) + \mu_3 (a, \theta_0 b, c) + \mu_3 (a, b, \theta_0 c).$
\end{itemize}
\end{defn}

We call a $2$-term $A_\infty$-algebra with a homotopy derivation a $2$-term $\mathrm{AssDer}_\infty$-pair. We denote such a pair by $(A_1 \xrightarrow{d} A_0, \mu_2, \mu_3, \theta_0, \theta_1, \theta_2).$ An $\mathrm{AssDer}_\infty$-pair is said to be skeletal if the underlying $2$-term $A_\infty$-algebra is skeletal, i.e. $d = 0$.

\begin{prop}\label{skeletal-chr}
There is a one-to-one correspondence between the set of all skeletal $\mathrm{AssDer}_\infty$-pairs and the set of all triples $(( A, \phi_A), (M, \phi_M), ( \theta, \psi))$, where $( A, \phi_A)$ is an AssDer pair, $(M, \phi_M)$ is a representation and $( \theta, \psi) \in C^3_{\mathrm{AssDer}} (A, M)$ is a $3$-cocycle in the cohomology of the AssDer pair with coefficients in $(M, \phi_M)$.
\end{prop}

\begin{proof}
Let $( A_1 \xrightarrow{0} A_0, \mu_2, \mu_3 , \theta_0, \theta_1, \theta_2)$ be a skeletal $\mathrm{AssDer}_\infty$-pair. Then it follows from Definition \ref{ass-der-inf-def}(a) that $\theta_0$ is a derivation for the associative algebra $(A_0, \mu_2).$ Moreover, the conditions (b) and (c) say that $(A_1, \theta_1)$ is a representation of the AssDer pair $(A_0, \theta_0)$. Finally, the condition (f) of Definition \ref{defn-2-ass-term} implies that $\delta_\mathrm{Hoch} (\mu_3)=0$ and condition (d) of Definition \ref{ass-der-inf-def} implies that $\delta_{\mathrm{Hoch}} \theta_2 + \delta \mu_3  = 0$. Therefore, $(\mu_3, - \theta_2) \in C^3_{\mathrm{AssDer}} (A_0, A_1)$ is a $3$-cocycle in the cohomology of the AssDer pair with coefficients in $(A_1, \theta_1)$.

Conversely, let $((A, \phi_A), (M, \phi_M), (\theta, \psi))$ be such a triple. Then it can be easily verify that $(M \xrightarrow{0} A, \mu_2 = ( \mu_A, l, r), \theta, \phi_A, \phi_M, - \psi)$ is a skeletal $\mathrm{AssDer}_\infty$-pair. The above correspondences are inverses to each other.
\end{proof}

\begin{defn}
A $2$-term $\mathrm{AssDer}_\infty$-pair $(A_1 \xrightarrow{d} A_0, \mu_2, \mu_3, \theta_0, \theta_1, \theta_2)$ is called strict if $\mu_3 = 0$ and $\theta_2 = 0.$
\end{defn}

\begin{exam}\label{exm-cr-m}
Let $(A, \phi_A)$ be an AssDer pair. Take $A_0 = A_1 = A$, $d = \mathrm{id}$, $\mu_2 = \mu$ (the associative multiplication on $A$), $\theta_0 = \theta_1 = \phi_A$. Then $(A_1 \xrightarrow{d} A_0, \mu_2, \mu_3 = 0, \theta_0, \theta_1, \theta_2 = 0)$ is strict $\mathrm{AssDer}_\infty$-pair.
\end{exam}

Next, we introduce crossed modules of AssDer pairs and show that strict $\mathrm{AssDer}_\infty$-pairs correspond to crossed modules of AssDer pairs.

\begin{defn}\label{crossed-mod}
A crossed module of AssDer pairs consist of a tuple $((A, \phi_A), (B, \phi_B), dt, \phi)$ in which $(A, \phi_A), (B, \phi_B)$ are both AssDer pairs, $dt : A \rightarrow B$ is a morphism of AssDer pairs and
\begin{align*}
\phi : B \otimes A \rightarrow A  \qquad \quad \phi : A \otimes B \rightarrow A
\end{align*}
defines an AssDer pair bimodule on $(A, \phi_A)$ satisfying the following conditions: for all $b \in B$ and $m, n \in A$,
\begin{itemize}
\item[(i)] $dt ( \phi (b, m)) = \mu_B ( b, dt(m))$,\\$dt ( \phi (m, b)) = \mu_B ( dt(m), b),$
\item[(ii)] $\phi ( dt(m), n ) = \mu_A (m, n),$\\$ \phi ( m, dt(n)) = \mu_A (m, n)$,
\item[(iii)] $\mu_A ( \phi (b, m), n) = \phi ( b, \mu_A (m, n))$,\\
$\mu_A (\phi (m, b), n ) = \mu_A( m , \phi (b, n)),$\\
$\phi ( \mu_A (m, n), b) = \mu_A ( m, \phi(n, b)),$
\item[(iv)] $\phi_A ( \phi (b,m)) = \phi ( \phi_B (b), m) + \phi (b, \phi_A (m)),$\\
$\phi_A ( \phi (m, b)) = \phi ( \phi_A (m), b) + \phi ( m, \phi_B (b)).$
\end{itemize}
\end{defn}

\begin{prop}\label{strict-chr}
There is a one-to-one correspondence between strict $AssDer_\infty$-pairs and crossed module of AssDer pairs.
\end{prop}

\begin{proof}
It is already known that strict $A_\infty$-algebras are in one-to-one correspondence with crossed modules of associative algebras \cite{das}. More precisely, $(A_1 \xrightarrow{d} A_0, \mu_2 , \mu_3 = 0)$ is a strict $A_\infty$-algebra if and only if $(A_1, A_0, d, \mu_2$) is a crossed module of associative algebras. Note that the associative products on $A_1$ and $A_0$ are respectively given by $\mu_{A_1} (m, n) := \mu_2 (dm, n) = \mu_2 (m, dn)$ and $\mu_{A_0} (a, b) = \mu_2 (a, b)$, for $m, n \in A_1$ and $a, b \in A_0$. It follows from (a) and (b) of Definition \ref{ass-der-inf-def}  that $\theta_1$ is a derivation on $A_1$ and $\theta_0$ is a derivation on $A_0$. Hence $(A_1, \theta_1)$ and $(A_0, \theta_0)$ are AssDer pairs. Since $\theta_0 \circ d = d \circ \theta_1$, we have $dt = d: A_1 \rightarrow A_0$ is a morphism of AssDer pairs. The conditions (i), (ii), (iii) of Definition \ref{crossed-mod} are also held.
Finally, the conditions (b) and (c) of Definition \ref{ass-der-inf-def}  are equivalent to the last condition of Definition \ref{crossed-mod}.
\end{proof}

The crossed module corresponding to the strict $\mathrm{AssDer}_\infty$-pair of Example \ref{exm-cr-m} is given by the tuple $((A, \phi_A), (A, \phi_A), \mathrm{id}, \mu_A)$.

\begin{defn}
Let $(A_1 \xrightarrow{d} A_0, \mu_2, \mu_3, \theta_0, \theta_1, \theta_2)$ and 
$(A'_1 \xrightarrow{d'} A'_0, \mu'_2, \mu'_3, \theta'_0, \theta'_1, \theta'_2)$ be two
$2$-term $\mathrm{AssDer}_\infty$-pair. A morphism between them consists of a morphism $(f_0, f_1, f_2)$ between the underlying $2$-term $A_\infty$-algebras together with a map $\mathcal{B} : A_0 \rightarrow A_1'$ such that the following conditions hold:
\begin{itemize}
\item[(i)] $\theta_0' ( f_0(a)) - f_0 ( \theta_0 (a)) = d' ( \mathcal{B} (a)),$
\item[(ii)] $\theta_1' ( f_1(m)) - f_1 ( \theta_1 (m)) =  \mathcal{B} (dm),$
\item[(iii)] $ f_1 ( \theta_2 (a, b)) + f_2 ( \theta_0 a, b) + f_2 ( a, \theta_0 b) - \theta_1' ( f_2 (a, b)) - \theta_2' ( f_0 (a), f_0 (b)) \\
= \mu_2' ( \mathcal{B}a, f_0 (b)) + \mu_2' ( f_0 (a), \mathcal{B} b) - \mathcal{B} ( l_2 (a, b)).$
\end{itemize}
\end{defn}

Let $A = ( A_1 \xrightarrow{d} A_0, \mu_2, \mu_3, \theta_0, \theta_1, \theta_2)$ and 
$A' = ( A_1' \xrightarrow{d'} A_0', \mu_2', \mu_3', \theta_0', \theta_1', \theta_2')$ be two $2$-term $\mathrm{AssDer}_\infty$-pairs and $ f = ( f_0, f_1, f_2, \mathcal{B})$ be a morphism between them. Let $A'' = ( A_1'' \xrightarrow{d''} A_0'', \mu_2'', \mu_3'', \theta_0'', \theta_1'', \theta_2'')$ be another $2$-term $\mathrm{AssDer}_\infty$-pair and $g = ( g_0, g_1, g_2, \mathcal{C})$ be a morphism from $A'$ to $A''$. Their composition is a morphism $ g \circ f : A \rightarrow A''$ of $\mathrm{AssDer}_\infty$-pairs whose components are given by $(g \circ f)_0 = g_0 \circ f_0,$ $(g \circ f)_1 = g_1 \circ f_1,$ and 
\begin{align*}
(g \circ f)_2 (a, b) = g_2 ( f_0 (a), f_0(b)) + g_1 ( f_2 (a, b)), ~~\quad \quad  \mathcal{D} = g_1 \circ \mathcal{B} + \mathcal{C} \circ f_0 : A_0 \rightarrow A_1''.
\end{align*}

For any $2$-term $\mathrm{AssDer}_\infty$-pair $A$, the identity morphism $\mathrm{id}_A$ is given by the identity chain map $A \rightarrow A$ together with $(\mathrm{id}_A)_2 = 0$ and $\mathcal{B} = 0$. The collection of $2$-term $\mathrm{AssDer}_\infty$-pairs and morphisms between them forms a category.
We denote this category by $2\mathrm{\bf AssDer}_\infty$.

\section{Categorification of AssDer pairs}\label{sec-cat}

%Categorification of Lie algebras was first studied by Baez and Crans \cite{baez}. The categorified Lie algebras are called Lie $2$-algebras and they are related to $2$-term strongly homotopy Lie algebras.

In this section, we study the categorification of AssDer pair, which we call AssDer $2$-pair. We show that the category of AssDer $2$-pairs and the category $2\mathrm{\bf AssDer}_\infty$ are equivalent.

A $2$-vector space $C$ is a category with vector space of objects $C_0$ and the vector space of arrows $C_1$ such that all structure maps in the category $C$ are linear. A morphism of $2$-vector spaces is a functor $F = (F_0, F_1)$ which is linear in the space of objects and arrows.  We denote the category of $2$-vector spaces by $2\mathrm{\bf Vect}$. Given a $2$-vector space $C = (C_1 \rightrightarrows C_0)$, we have a $2$-term complex $\mathrm{ker} (s) \xrightarrow{t} C_0$. A morphism between $2$-vector spaces induces a morphism between $2$-term complexes. Conversely, any $2$-term complex $A_1 \xrightarrow{d} A_0$ gives rise to a $2$-vector space $\mathbb{V} = ( A_0 \oplus A_1 \rightrightarrows A_0)$ in which  the set of objects is $A_0$ and the set of morphisms is $A_0 \oplus A_1$. The structure maps are given
by $s(a, m) = a,~ t(a, m) = a + dm$ and $i(a) = (a, 0)$. A morphism between $2$-term chain complexes
induces a morphism between corresponding $2$-vector spaces.
We denote the category of $2$-term complexes by $2\mathrm{\bf TermCom}$. There is an equivalence of categories
$2\mathrm{\bf TermCom} \simeq  2\mathrm{\bf Vect}.$

\begin{defn}
An associative $2$-algebra 
is a $2$-vector space $C$ equipped with
a bilinear functor $\mu : C \otimes C \rightarrow C$
and a trilinear natural isomorphism called the associator
    \begin{center}
    $\mathcal{A}_{\xi, \eta, \zeta} : \mu ( \mu(\xi, \eta), \zeta  ) \rightarrow   \mu   (   \xi, \mu (\eta, \zeta)) $
    \end{center}
satisfying the following identity represented by the pentagon
%\[
%\xymatrixrowsep{0.5in}
%\xymatrixcolsep{0.1in}
%\xymatrix{
%    &    \mu \big(~ \mu ( \mu(\xi , \eta) , \zeta ) ,~ \lambda \big) \ar[ld]_{\mathcal{A}_{\xi, \eta, \zeta}} \ar[rd]^{\mathcal{A}_{\mu(\xi, \eta), \zeta, \lambda}} &  \\
%\mu \big(~ \mu (\xi , \mu (\eta , \zeta)) ,~ \lambda \big) \ar[d]_{\mathcal{A}_{\xi, \mu (\eta, \zeta), \lambda}} & & \mu \big(~    \mu(\xi , \eta) , ~ \mu  (\zeta ,  \lambda ) ~\big) \ar[d]^{\text{id}} \\
%\mu \big(   \xi , ~\mu (  \mu (\eta , \zeta) , \lambda) ~\big) \ar[rd]_{\mathcal{A}_{\eta, \zeta, \lambda}} & & \mu \big(~ \mu (\xi , \eta) , ~  \mu(\zeta , \lambda)~ \big) \ar[ld]^{~~\mathcal{A}_{\xi, \eta, \mu(\zeta, \lambda)}} \\
% & \mu \big( \xi , ~ \mu ( \eta , \mu (\zeta , \lambda) ) ~\big). &  \\
%}
%\]

\[
\xymatrixrowsep{0.5in}
\xymatrixcolsep{-0.5in}
\xymatrix{
 & & \mu \big( \mu ( \mu(\xi , \eta) , \zeta ) , \lambda \big) \ar[lld]_{\mathcal{A}_{\xi, \eta, \zeta}} \ar[rrd]^{\mathcal{A}_{\mu(\xi, \eta), \zeta, \lambda}}  &&\\
\mu \big( \mu (\xi , \mu (\eta , \zeta)) , \lambda \big) \ar[rd]_{\mathcal{A}_{\xi, \mu (\eta, \zeta), \lambda}} &&&& \mu \big(  \mu(\xi , \eta) ,  \mu  (\zeta ,  \lambda ) \big) \ar[ld]^{~~\mathcal{A}_{\xi, \eta, \mu(\zeta, \lambda)}} \\
 & \mu \big(  \xi , \mu (  \mu (\eta , \zeta) , \lambda) \big) \ar[rr]_{\mathcal{A}_{\eta, \zeta, \lambda}} && \mu \big( \xi ,  \mu ( \eta , \mu (\zeta , \lambda) ) \big). & 
}
\]
\end{defn}

\begin{defn}
A morphism between associative $2$-algebras $(C, \mu, \mathcal{A})$ and $(C', \mu', \mathcal{A}')$ consists of a functor $F = (F_0, F_1)$ from the underlying $2$-vector space $C$ to $C'$ and a bilinear natural transformation 
\begin{align*}
F_2 (\xi, \eta) : \mu' ( F_0 (\xi), F_0 (\eta)) \rightarrow F_0 ( \mu (\xi, \eta))
\end{align*}
such that the following diagram commutes
\begin{align*}
\xymatrixrowsep{0.5in}
\xymatrixcolsep{-0.5in}
\xymatrix{
 &    \mu' \big( \mu' ( F_0 (\xi) , F_0 (\eta) ) , F_0 (\zeta) \big) \ar[rd]^{F_2 (\xi, \eta)} \ar[ld]_{\mathcal{A}'_{F_0 (\xi), F_0 (\eta), F_0 (\zeta)}} &  \\
\mu' \big(  F_0 (\xi) , \mu' (   F_0(\eta) , F_0 (\zeta))\big) \ar[d]_{F_2 (\eta, \zeta)} & &  \mu' \big(F_0 (\mu(\xi , \eta)) , F_0 (\zeta) \big)  \ar[d]^{F_2 \big(\mu(\xi, \eta), \Phi (\zeta) \big)} \\
\mu' \big( F_0 (\xi) , F_0 (\mu(\eta , \zeta)) \big) \ar[rd]_{F_2 \big( \xi, \mu (\eta, \zeta) \big)} & &  F_0 \big( \mu( \mu (\xi , \eta) , \zeta ) \big) \ar[ld]^{\mathcal{A}_{\xi, \eta, \zeta}} \\
  & F_0 \big( \mu ( \xi , \mu(\eta , \zeta) ) \big). & \\
}
\end{align*}
\end{defn}

The composition of two associative $2$-algebra morphisms is again an associative $2$-algebra morphism. More precisely, let $C, C'$ and $C''$ be three associative $2$-algebras and $F : C \rightarrow C'$, $G: C' \rightarrow C''$ be associative $2$-algebra morphisms. Then their composition $G \circ F : C \rightarrow C''$ is an associative $2$-algebra morphism given by
$ (G \circ F)_0 = G_0 \circ F_0 , ~ (G \circ F)_1 = G_1 \circ F_1$
and $(G \circ F)_2$ is given by the composition
\[
\xymatrixrowsep{3in}
\xymatrix{
\mu'' ( G_0 \circ F_0 (\xi), G_0 \circ F_0 (\eta))  \ar[r]^{G_2 (F_0 (\xi), F_0 (\eta) )} & G_0 ( ~ \mu'   ( F_0 (\xi) , F_0 (\eta) ) ~ )  \ar[r]^{G_0 (F_2 (\xi, \eta))}& (G_0 \circ F_0) (\mu (\xi, \eta)). \\
}
\]
Finally, for any associative $2$-algebra $C$, the identity morphism $\text{id}_C : C \rightarrow C$ is given by the identity functor as its linear functor and the identity natural transformation as $(\text{id}_C)_2$.
We denote the category of associative $2$-algebras and morphisms between them by $\mathrm{\bf Ass}2$.

\begin{defn}
Let $(C, \mu, \mathcal{A})$ be an associative $2$-algebra. A $2$-derivation on it consists of a linear functor $D: C \rightarrow C$ and a bilinear natural isomorphism, called the derivator
\begin{align*}
\mathcal{D}_{a, b} : D (\mu (a, b)) \rightsquigarrow \mu (Da, b) + \mu (a, Db)
\end{align*}
satisfying the following
\[
\xymatrixrowsep{0.5in}
\xymatrixcolsep{-.5in}
\xymatrix{
    &    D ( (a \cdot b) \cdot c) \ar[ld]_{\mathcal{D}_{a \cdot b, c}} \ar[rd]^{\mathcal{A}_{a,b,c}} &  \\
(a \cdot b) \cdot D(c) + D(a \cdot b) \cdot c \ar[d]_{\mathcal{D}_{a,b}} & &  D (a \cdot (b \cdot c)) \ar[d]^{\mathcal{D}_{a, b \cdot c}} \\
(a \cdot b) \cdot D(c) + (a \cdot D(b) + D(a) \cdot b) \cdot c  \ar[rd]_{\mathcal{A}_{a, b, Dc} + \mathcal{A}_{a, Db, c} + \mathcal{A}_{Da, b, c}} & & a \cdot D(b \cdot c) + D(a) \cdot (b \cdot c) \ar[ld]^{\mathcal{D}_{b,c}} \\
 &a \cdot (b \cdot D(c)) + a \cdot ( D(b) \cdot c) + D(a) \cdot (b \cdot c). &  \\
}
\]
\end{defn}

In the above diagram, we use the notation $\mu (a, b) = a \cdot b.$
We call an associative $2$-algebra together with a $2$-derivation by an $\mathrm{AssDer}2$-pair.

\begin{defn}
Let $(C, \mu, \mathcal{A}, D, \mathcal{D})$ and $(C', \mu', \mathcal{A}', D', \mathcal{D}')$ be two $\mathrm{AssDer}2$-pairs. A morphism between them consists of an associative $2$-algebra morphism $(F = (F_0, F_1), F_2)$ and a natural isomorphism
%\begin{align*}
$\Phi_a : D' \circ F_0 ( a) \rightarrow F_0 \circ D ( a)$
%\end{align*}
such that the following diagram commutes
\[
\xymatrixrowsep{0.5in}
\xymatrixcolsep{-0.1in}
\xymatrix{
    &    D' ( F_0 (a) \cdot' F_0 (b)) \ar[ld]_{\mathcal{D}'} \ar[rd]^{\mathcal{A}_{a,b,c}} &  \\
D' (F_0 (a)) \cdot' F_0 (b) + F_0 (a) \cdot' D' (F_0 (b)) \ar[d]_{\Phi} & &  D' (F_0 ( a \cdot b)) \ar[d]^{\Phi} \\
F_0 (D(a) \cdot' F_0 (b) + F_0 (a) \cdot' F_0 (D(a))  \ar[rd]_{F_2} & & F_0 (D( a \cdot b)) \ar[ld]^{\mathcal{D}} \\
 &F_0 ( D(a) \cdot b + a \cdot D(b)). &  \\
}
\]
\end{defn}

Here we use the bifunctors $\mu$ and $\mu'$ as $\cdot$ and $\cdot'$ respectively.
We denote the category of  $\mathrm{AssDer}2$-pairs together with morphisms between them by $\mathrm{\bf AssDer}2$.

It is known that the categories $2 \mathrm{\bf A}_\infty$ and $\mathrm{\bf Ass}2$ are equivalent. See for example \cite{das}. A functor $T : 2 \mathrm{\bf A}_\infty \rightarrow \mathrm{\bf Ass}2$ is given as follows.
Let $A = ( A_1 \xrightarrow{d} A_0, \mu_2, \mu_3)$ be a $2$-term $A_\infty$-algebra. The corresponding associative $2$-algebra is defined on the $2$-vector space $A_0 \oplus A_1 \rightrightarrows A_0$. The bifunctor $\mu$ and the associator $\mathcal{A}$ is given by
\begin{align*}
\mu ((a, m), (b,n)) =~& ( \mu_2 (a, b), \mu_2 (a, n) + \mu_2 (m, b) + \mu_2 (dm, n)),\\
\mathcal{A}_{a, b, c} =~~& ((ab)c, \mu_3 (a,b,c)).
\end{align*}
For any $A_\infty$-algebra morphism $(f_0, f_1, f_2)$ from $A$ to $A'$, the associative $2$-algebra morphism from $T(A)$ to $T(A')$ is given by
\begin{align*}
F_0 = f_0, \qquad F_1 = f_1 \quad \text{ and } \quad F_2 (a, b) = ( \mu' ( f_0 (a), f_0 (b)), f_2 (a, b)).
\end{align*}
On the other hand, a functor $S : \mathrm{\bf Ass}2 \rightarrow 2 \mathrm{\bf A}_\infty$
is given as follows. Given an associative $2$-algebra $C = ( C_1 \rightrightarrows C_0, \mu, \mathcal{A}),$ the corresponding $2$-term $A_\infty$-algebra is defined on the complex $A_1 = \mathrm{ker} s \xrightarrow{ d = t|_{\mathrm{ker} s}} C_0 = A_0$. Define $\mu_2 : A_i \otimes A_j \rightarrow A_{i+j}$ and $\mu_3 : A_0 \otimes A_0 \otimes A_0 \rightarrow A_1$ by
\begin{align*}
\mu_2(a, b) = \mu(a, b), ~ \mu_2 (a, m) =& \mu (i(a), m), ~ \mu_2 (m, a) = \mu_2 (m, i(a)),~  \mu_2 (m, n) = 0, \\
 \text{ and }~\mu_3 (a, b, c) =~& \mathcal{A}_{a, b, c} - i ( s(\mathcal{A}_{a, b, c})). 
\end{align*} 
For any associative $2$-algebra morphism $(F_0, F_1, F_2) : C \rightarrow C'$, the corresponding $A_\infty$-algebra morphism from $S(C)$ to $S (C')$ is given by
\begin{align*}
f_0  = F_0, \qquad f_1 = F_1|_{\mathrm{ker } s}, \qquad f_2 (a, b) = F_2 (a, b) - i (s F_2 (a, b)).
\end{align*}
It is well-known that the above two functors provide the equivalence between $2 \mathrm{\bf A}_\infty$ and $\mathrm{\bf Ass}2$ (see for instance, \cite{das}). This equivalence can be extended to respective categories equipped with derivations. More precisely, we have the following.

\begin{thm}\label{final-thrm}
The categories $2\mathrm{\bf AssDer}_\infty$ and $\mathrm{\bf AssDer}2$ are equivalent.
\end{thm}

\begin{proof}
Given a $2$-term $\mathrm{AssDer}_\infty$-pair $(A_1 \xrightarrow{d} A_0, \mu_2, \mu_3, \theta_0, \theta_1, \theta_2)$, consider the corresponding associative $2$-algebra $T(A)$. A functor $D: T(A) \rightarrow T(A)$ and the derivator $\mathcal{D}$ is given by
\begin{align*}
D ((a, m)) = ( \theta_0 (a), \theta_1 (m)), \qquad \mathcal{D}_{a, b} = ( ab, \theta_2 (a, b)). 
\end{align*}
It can be checked that if $(f_0,f_1, f_2, \mathcal{B})$ is a morphism of $\mathrm{AssDer}_\infty$-pairs, then $(F_0, F_1, F_2, \Phi)$ is a morphism of corresponding $\mathrm{AssDer}2$-pairs, where $\Phi(a) = \mathcal{B} (a).$

Conversely, for any $\mathrm{AssDer}2$-pair $(C, \mu, \mathcal{A}, D, \mathcal{D})$, consider the $2$-term $A_\infty$-algebra $S(C)$. A homotopy derivation on $S(C)$ is given by $\theta_0 = D (i(a)),~ \theta_1 (m) = D|_{\mathrm{ker} s} (m)$ and $\theta_2 (a, b) = \mathcal{D}_{a, b} - i( s \mu (a,b)).$ If $(F_0, F_1, F_2, \Phi)$ is a morphism of $\mathrm{AssDer}2$-pairs, then $(f_0, f_1, f_2, \mathcal{B})$ is a morphism between corresponding $2$-term $\mathrm{AssDer}_\infty$-pairs, where $\mathcal{B} (a) = \Phi(a).$

Thus, it remains to prove that the composition $T \circ S$ is naturally isomorphic to the identity functor $1_{\mathrm{ AssDer}2}$ and $S \circ T$ is naturally isomorphic to $1_{ 2 \mathrm{\bf AssDer}_\infty}$. For any $\mathrm{AssDer}2$-pair $(C, \mu, \mathcal{A}, D, \mathcal{D})$, the AssDer$2$-pair structure on $(T \circ S) (C)$ is defined on the $2$-vector space $A_0 \oplus A_1 \rightrightarrows A_0$, where $A_0 = C_0$ and $A_1 = \mathrm{ker }s.$ Define $\theta : T \circ S \rightarrow 1_{\mathrm{\bf AssDer}2}$ by $\theta_C : (T \circ S) (C) \rightarrow 1_{\mathrm{\bf AssDer}2} (C)$ with $(\theta_C)_0 (a) =a$, ~$(\theta_C)_1 (a, m) = i(a) + m$. Then $\theta_C$ is an isomorphism of AssDer$2$-pairs. It is also a natural isomorphism.

For any $2$-term $\mathrm{ AssDer}_\infty$-pair $A = (A_1 \xrightarrow{d} A_0, \mu_2, \mu_3, \theta_0, \theta_1, \theta_2)$, the
$\mathrm{AssDer}_\infty$-pair structure on $(S \circ  T) (A)$ is defined on the same complex $A_1 \xrightarrow{d} A_0$. In fact, we get back the same $2$-term $\mathrm{ AssDer}_\infty$-pair. Therefore, the natural isomorphism $\vartheta: S \circ T \rightarrow 1_{2 \mathrm{\bf AssDer}_\infty}$ is given by the identity. 
\end{proof}

\begin{exam}
    Let $(A, \phi_A)$ be an AssDer pair. Then $(A \xrightarrow{\mathrm{id}} A, \mu_2 = \mu, \mu_3 = 0, \theta_0 = \theta_1 = \phi_A, \theta_2 = 0)$ is a strict $\mathrm{AssDer}_\infty$-pair. The corresponding $\mathrm{AssDer2}$-pair is also strict in the sense that the associator and the derivator are both trivial.
\end{exam}
%\medskip

\section* {Conclusions.} In this paper, we mainly concentrate on a pair of associative algebra and a derivation on it. We call such a pair an AssDer pair. Among other things, we study central extensions and deformations of an AssDer pair by extending the classical extensions and deformations of associative algebras. For this, we define a cohomology theory for AssDer pairs generalizing the cohomology of LieDer pairs introduced in \cite{tang-f-sheng}.

In \cite{das2} the author studies deformations of multiplications in a nonsymmetric operad which generalizes the deformation of associative algebras. As applications, the author formulates deformation of various Loday-type (e.g. dendriform, tridendriform, dialgebra, quadri) algebras. See \cite{das2} for explicit cohomology of Loday-type algebras. Given a nonsymmetric operad $\mathcal{O}$ with a multiplication $m \in \mathcal{O}(2)$ (i.e. $m$ satisfies $m \circ_1 m = m \circ_2 m$), an element $\phi \in \mathcal{O}(1)$ is called a derivation for $m$ if $\phi$ satisfies
%\begin{align*}
$$\phi \circ m = m \circ_1 \phi + m \circ_2 \phi.$$
%\end{align*}
By the method of the present paper, one may study deformations of a pair $(m , \phi)$, where $m$ is a multiplication on $\mathcal{O}$ and $\phi$ is a derivation for $m$. Therefore, one may deduce deformations of Loday-type algebras equipped with derivations. On the other hand, Balavoine \cite{bala} studied deformations of algebras over quadratic operads. Derivations on an algebra over an operad are studied in \cite{loday, doubek-lada}. In a subsequent paper, we aim to construct an explicit cohomology and deformation theory for $\mathcal{P}$-algebras equipped with derivations, where $\mathcal{P}$ is a quadratic operad. The results of the present paper can be dualized to study deformations of coalgebras with coderivations. Since an $A_\infty$-algebra can be described by a square-zero coderivation on the tensor coalgebra of a graded vector space, one can explore formal deformations of $A_\infty$-algebras and compare with the results of \cite{fia-pen}.

%\medskip

In \cite{cra-moer} the authors considered deformations of a Lie algebroid $A$. Such deformations are governed by a (shifted) graded Lie algebra on the space of multiderivations on $A$. A $1$-cocycle of the corresponding complex is given by a Lie algebroid derivation on $A$. A derivation of the underlying vector bundle $A$ is a Lie algebroid derivation if it is also a derivation for the Lie bracket on $\Gamma A$. 
%For any $X \in \Gamma A$, the map $\Phi_X : \Gamma A \rightarrow \Gamma A$ is a Lie algebroid derivation. 
Lie algebroid derivations are worth interesting as their flows give rise to Lie algebroid automorphisms. It would be interesting to study deformations of Lie algebroids equipped with Lie algebroid derivation.

\medskip

\noindent {\bf Acknowledgements.} The authors thank the esteemed referee for his/her comments on the earlier version of the manuscript. The research of Apurba Das was supported by the postdoctoral fellowship of Indian Institute of Technology Kanpur. Ashis Mandal acknowledges the support from MATRICS-Research grant  MTR/2018/000510 by SERB, DST, Government of India.

%\medskip

%\noindent {\bf Data Availability Statement.} NA.

%\mbox{ }\\
%
%\providecommand{\bysame}{\leavevmode\hbox to3em{\hrulefill}\thinspace}
%\providecommand{\MR}{\relax\ifhmode\unskip\space\fi MR }
% \MRhref is called by the amsart/book/proc definition of \MR.
%\providecommand{\MRhref}[2]{%
 % \href{http://www.ams.org/mathscinet-getitem?mr=#1}{#2}
%}
%\providecommand{\href}[2]{#2}

\end{document}